\documentclass{article}
\usepackage[utf8]{inputenc}
\usepackage{geometry}
\usepackage{amsmath}
\usepackage{amssymb}
\usepackage{amsthm}
\usepackage{tikz}
\usetikzlibrary{arrows.meta, positioning, calc, decorations.pathreplacing, shapes.misc}

\usepackage{xeCJK}
\usepackage{ulem}

\usepackage{matlab-prettifier}

\usepackage{mathtools}
\usepackage{diagbox}
\usepackage{float}
\usepackage[title]{appendix}

\usepackage{fancyhdr}
\usepackage{graphicx}
\usepackage{subfig}
\usepackage[colorlinks,            linkcolor=black,            anchorcolor=black,            citecolor=black            ]{hyperref}
\graphicspath{ {./images/} }
\newtheorem{proposition}{Proposition}

\newtheorem{remark}{Remark}

\newcommand{\p}{\partial}

\newcommand{\eps}{\epsilon}

\mathtoolsset{showonlyrefs}

\DeclareMathOperator\init{init}

\numberwithin{equation}{section}

\title{Resolving the Blow-Up: A Time-Dilated Numerical Framework for Multiple Firing Events in Mean-Field Neuronal Networks}

\author{Xu'an Dou\thanks{Beijing International Center for Mathematical Research, Peking University, Beijing, 100871, China (dxa@pku.edu.cn)},\, \, Louis Tao \thanks{Center for Bioinformatics, School of Life Sciences, and Center for Quantitative Biology, Academy for Advanced Interdisciplinary Studies, Peking University, Beijing, 100871, China (taolt@mail.cbi.pku.edu.cn)} \, \, Zhe Xue \thanks{Institute for Theoretical Sciences, Westlake University, Hangzhou, Zhejiang Province, 310030, China } \, \,and\, \,  Zhennan Zhou\thanks{Institute for Theoretical Sciences, Westlake University, Hangzhou, Zhejiang Province, 310030, China (zhouzhennan@westlake.edu.cn).}}

\date{\today}

\begin{document}

\maketitle

\begin{abstract}
In large-scale excitatory neuronal networks, rapid synchronization manifests as {multiple firing events (MFEs)}, mathematically characterized by a finite-time blow-up of the neuronal firing rate in the mean-field Fokker-Planck equation. Standard numerical methods struggle to resolve this singularity due to the divergent boundary flux and the instantaneous nature of the population voltage reset. In this work, we propose a robust {multiscale numerical framework based on time dilation}. By transforming the governing equation into a dilated timescale proportional to the firing activity, we desingularize the blow-up, effectively stretching the instantaneous synchronization event into a resolved mesoscopic process. This approach is shown to be physically consistent with the {microscopic cascade mechanism} underlying MFEs and the system's inherent fragility. To implement this numerically, we develop a hybrid scheme that utilizes a {mesh-independent flux criterion} to switch between timescales and a semi-analytical ``moving Gaussian'' method to accurately evolve the post-blowup Dirac mass. Numerical benchmarks demonstrate that our solver not only captures steady states with high accuracy but also efficiently reproduces periodic MFEs, matching Monte Carlo simulations without the severe time-step restrictions associated with particle cascades.
\end{abstract}



\section{Introduction}

Mathematical modeling of large-scale neuronal networks bridges the stochastic dynamics of individual neurons and the collective emergent behaviors of the brain. Among various approaches, the Integrate-and-Fire (IF) neuron model remains a standard abstraction, favored for its balance between analytical tractability and biological relevance \cite{brunel1999fast, caceres2011analysis}.

In the limit of a large number of neurons $N_p \to \infty$, the discrete interactions of a homogeneous network are effectively captured by macroscopic quantities: the probability density of the membrane potential $p(v,t)$ and the global network firing rate $N(t)$. These are governed by the mean-field Fokker-Planck equation \cite{caceres2011analysis, delarue2015particle}:
\begin{equation}\label{eq:intro-fp}
    \begin{cases}
    \p_{t}p + \p_v \left( \left[ -v + bN(t) \right] p \right) = a \p_{vv}p + N(t)\delta(v - V_R), & v < V_F, \\
    p(t, V_F) = 0, \\
    N(t) = -a \p_v p(t, V_F).
    \end{cases}
\end{equation} 
Here, $a > 0$ represents the diffusion coefficient arising from background noise, and $b$ quantifies the strength of synaptic, network connectivity. In this work, we focus on the case $b>0$, corresponding to an excitatory network. The dynamics of each IF neuron is confined to the region $v < V_F$, where $V_F$ is the firing threshold. Upon reaching this threshold, individual neurons {fire} (or {spike}) and are instantaneously reset to the resting potential $V_R$ (with $V_R < V_F$), a process captured by the source term $N(t)\delta(v - V_R)$ and the boundary condition. 

We view the mean-field equation \eqref{eq:intro-fp} as a macroscopic model at the population level; in contrast, a microscopic model describes the membrane potential of each individual neuron within a network. Macroscopic models have many advantages; however, their applicability is often considered limited. While equation \eqref{eq:intro-fp} successfully captures steady states and mild oscillations \cite{brunel1999fast,brunel2000dynamics}, it is unclear whether it is still a valid description for highly synchronized neuronal activity. Indeed, in the computational neuroscience literature, the Fokker-Planck equation \eqref{eq:intro-fp} is often derived formally by assuming that the number of neurons is large and that their activities are ``uncorrelated''. In an apparent contrast, during a synchronization event,  a non-trivial fraction of the neuronal population spikes within an infinitesimal time window. Thus, it remains a puzzle whether traditional macroscopic approaches are still valid in this regime.

Scientifically, such synchronization---often referred to as a \textit{multiple firing event} (MFE)---is fundamental to biological processes, including rhythmic motor control and sensory information processing \cite{rangan2013emergent}. However, capturing this emergent behavior poses a severe challenge for computational modeling. The traditional Fokker-Planck approach \eqref{eq:intro-fp} relies heavily on the assumption of uncorrelated neuronal activity, an assumption that is commonly believed to fail during the instantaneous cascades \cite{rangan2013dynamics, zhang2014distribution}. Consequently, the breakdown of standard macroscopic descriptions in this synchronized regime has motivated a wealth of alternative computational methods and theoretical investigations. These widespread efforts range from direct stochastic particle simulations \cite{CiCP-30-820} to specialized coarse-grained network models \cite{zhang2014coarse}.

Mathematically, a remarkable property of \eqref{eq:intro-fp} is that the firing rate $N(t)$ can blow up in finite time, as first discovered in \cite{caceres2011analysis}, and further developed in \cite{Antonio_Carrillo_2015,roux2021towards}. Unlike the interior mass-concentration blow-ups in, e.g., Keller-Segel systems, the singularity here is in the boundary flux at the firing threshold $V_F$. As the population density accumulates near $V_F$, the spatial derivative $-\p_vp(t,V_F)$ diverges, and the Dirichlet boundary condition $p(t, V_F) = 0$ might no longer hold \cite{munoz2025free}. The blow-up of $N(t)$ turns out to be closely related to the synchronization. However, it is established in \cite{carrillo2013CPDEclassical} that the classical solutions of \eqref{eq:intro-fp} exists if and only if $N(t)$ remains bounded. Thus, to describe synchronization within this framework, we need to define \textit{generalized solutions} of \eqref{eq:intro-fp} that allows the blow-up of $N(t)$ as part of the dynamics, instead of its cessation.

The first generalized solution of \eqref{eq:intro-fp} is proposed in \cite{delarue2015particle} from the probabilistic perspective, based on the corresponding McKean-Vlasov SDE of \eqref{eq:intro-fp}. It is also established in \cite{delarue2015particle} a rigorous limit from a microscopic particle system of neurons to the macroscopic mean-field description, in which we clearly see the correspondence between the blow-up of $N(t)$ and the synchronization. The generalized solution involves a criterion for the synchronization size, to which a formula of similar spirits also appear in the computational work \cite{zhang2014distribution}. Notably, similar mathematical mechanisms appear in the modeling of other disciplines: the systemic risk in finance and the instant freezing of the supercooled water.  There are many subsequent developments following \cite{delarue2015particle}  both in theoretical analysis \cite{2019-AAP1403,hambly2019mckean,2020BLMSLedgerAndreas,delarue2022global,bayraktar2024mckean,chu2025supercooled} and in numerical analysis \cite{cuchiero2024implicit,kaushansky2019simulation,lipton2021semi,kaushansky2023convergence}. Those works are mostly focusing on probabilistic approaches and devoted to the aforementioned models from finance and physics, which are simpler than \eqref{eq:intro-fp} in the sense that there is no reset mechanism. Indeed, the reset term is a unique difficulty for models in neuroscience. 

From the PDE perspective, generalized solutions have been proposed by introducing a new timescale -- the dilated timescale, which relates to the firing rate $N(t)$ as follows
\begin{equation}\label{def:tau-intro}
    d\tau =N(t)dt.
\end{equation} The blow-up of $N(t)$ can be interpreted as a Dirac mass in time in the original timescale. Thus, by \eqref{def:tau-intro} a single blow-up time in $t$ can be dilated into an interval in the new timescale $\tau$. In this way, the singularity in $t$ can be resolved by a regular dynamics in $\tau$. This generalization idea is proposed independently in \cite{dou2022dilating} and \cite{taillefumier2022characterization,sadun2022global} for different neuroscience models, and is further developed in \cite{dou2024noisy,papadopoulos2025physical}, see also \cite{carrillo2024classical}. 

To summarize, on the one hand, there is an increasing interest in computational neuroscience regarding synchronization and MFEs, a regime where the traditional Fokker-Planck approach fails. On the other hand, it is increasingly clear in the theoretical literature that synchronization corresponds to the blow-up of $N(t)$; while classical solutions cannot resolve this singularity, generalized solutions (which might be constructed via time-dilation) can. Consequently, a nontrivial gap emerges between theoretical analysis and computational practice.

In this work, we bridge this gap by developing the first multiscale numerical framework for \eqref{eq:intro-fp} in which the blow-ups are resolved physically by the time-dilation approach \eqref{def:tau-intro}. This serves as a first step toward the long-term goal of systematically integrating the time-dilation idea into practical computation tools for more complicated neuronal networks. Our approach resolves the time singularity around synchronization by introducing a more regular dynamics in the $\tau$ timescale. This framework acts as a mesoscopic bridge that seamlessly connects with the macroscopic model \eqref{eq:intro-fp}, allowing us to resolve MFEs without retreating to microscopic particle systems (in contrast to \cite{CiCP-30-820,DXZ2024}).

To build our framework, our first step is to theoretically clarify the \textit{physical} continuation of \eqref{eq:intro-fp} after a blow-up. Specifically, the generalized solution shall be consistent with an underlying microscopic particle system that is a faithful representation of realistic neuronal dynamics. To achieve this, we explicitly incorporate the effect of an (infinitesimal) refractory period \cite{deville2008synchrony,caceres2014beyond,petronilia2024contagious}. Incorporating this mechanism precludes the unphysical phenomenon of eternal blow-up \cite{dou2022dilating,dou2024noisy}, enables the continuation of the solution even in the strongly excitatory regime ($b \gg 1$), and gives rise to rich dynamics characterized by periodic blow-up solutions. However, it also introduces an additional level of discontinuity that poses theoretical challenges, partially explaining its omission in several prior studies \cite{delarue2015particle,dou2022dilating,dou2024noisy}. We note that generalized solutions incorporating refractory periods have been studied in \cite{taillefumier2022characterization,sadun2022global,papadopoulos2025physical} for a closely related but different macroscopic model.

To unify our understanding of synchronization, we then examine the resolution of the macroscopic blow-up through three interrelated multiscale perspectives: the microscopic particle system, the mesoscopic time-dilation dynamics in the $\tau$ timescale, and the fragility criteria \cite{hambly2019mckean}. Although these processes operate from distinct conceptual standpoints, we demonstrate that they are theoretically equivalent in resolving the singularity. This unified viewpoint not only guarantees the physical and mathematical validity of our approach but also naturally highlights the time-dilation formulation as a convenient platform for developing numerical tools.

Building upon this theoretical foundation, we propose an adaptive multiscale numerical framework. The solver dynamically couples the macroscopic PDE evolution with the mesoscopic resolution of MFEs. To reliably detect the onset of synchronization ($N=\infty$) during the computation, we establish a novel criterion based on the loss of the Dirichlet boundary condition \cite{munoz2025free}. This intrinsic, mesh-independent formulation avoids the empirical thresholds in previous works \cite{DXZ2024}. Moreover, we rigorously establish the theoretical properties necessary to ensure that the transitions between the two timescales are mathematically consistent. Furthermore, the incorporation of the refractory period induces a post-blowup singularity: the reset population emerges as a Dirac mass, which is incompatible with a fixed-grid discretization. To resolve this, we introduce a semi-analytical ``moving Gaussian'' scheme. This localized treatment allows the newly reset population to evolve according to an effective dynamics until it is sufficiently regularized, at which point it can be accurately projected back onto the Eulerian macroscopic grid.

To validate the effectiveness of the proposed framework, we conduct a series of numerical experiments across different dynamical regimes. Beyond verifying the solver against classical steady states and single blow-up scenarios, we demonstrate its capability to capture complex, long-time periodic synchronization patterns. A highlight of our approach is its superior computational efficiency compared to direct microscopic particle simulations. While a faithful simulation of the $N_p$-particle system typically necessitates a restrictive time step significantly smaller than $1/N_p$ to resolve the fine-grained interactions during a cascade, our mean-field solver requires only moderately small discretization parameters to accurately capture the onset and size of synchronization events. This robustness stems from the fact that the time-dilation reformulates the singular event into a regularized, trackable process, effectively bypassing the severe time-step constraints inherent in the stiff, microscopic dynamics. Our results show that the mean-field model, when equipped with this numerical continuation, serves as a highly efficient proxy for the collective, synchronizing behavior of large-scale neuronal populations.

The remainder of this paper is organized as follows. In Section \ref{sec:theory}, we present the theoretical multiscale framework and provide the unifying interpretation of the three different perspectives on MFEs. In Section \ref{sec:algorithm}, we detail the numerical algorithms, focusing on the transition between timescales, and the semi-analytical handling of the post-blowup profile. Finally, Section \ref{sec:experiments} is devoted to a comprehensive set of numerical experiments, including convergence studies and the exploration of diverse dynamical regimes.

\section{Resolving Blow-ups Physically: a Multiscale Framework}\label{sec:theory}

In this section, we present a theoretical framework for solving the mean-field Fokker-Planck equation \eqref{eq:intro-fp} in the presence of blow-ups. For convenience we recall the equation here 
\begin{align}\label{eq:fp-classical}
    \p_{t}p + \p_v([-v+bN(t)]p) &= a\p_{vv}p + N(t)\delta_{v=V_R},\qquad &t>0, v<V_F, \\
    p(t,V_F) &= 0,\qquad &t>0, \label{zerobc-classical} \\
    N(t) &= -a\p_vp(t,V_F)\geq 0,\qquad &t>0, \label{def-N-classical} \\ 
    p(t=0,v) &= p_{\init}(v),\qquad &v\leq V_F. \label{ic:fp-classical}
\end{align} 

We focus on the excitatory case $b>0$ as the solution can blow up in this scenario \cite{caceres2011analysis,roux2021towards}. In the inhibitory case $b<0$, the system is known to be globally well-posed \cite{carrillo2013CPDEclassical,Antonio_Carrillo_2015}.

\subsection{A General Framework}

Suppose the firing rate blows up at time $t$, i.e., $N(t^-)=+\infty$, the classical solution of the PDE \eqref{eq:intro-fp} stops to exist \cite{carrillo2013CPDEclassical}. At the microscopic level this means the neuronal network synchronizes at $t$. To physically continue the dynamics, we resolve the singularity according to the following three-step procedure:
\begin{enumerate}
    \item Determine a \textit{blow-up size} $\Delta m$ according to the \textit{pre-blowup profile} $p(t^-)$. The blow-up size shall physically correspond to the synchronization size, which is the proportion of the neuron population that spike synchronously at $t$. 
    \item Obtain the physical \textit{post-blowup profile} according to $\Delta m$ and $p(t^-)$
    \begin{equation}\label{post-blowup-calculation-0}
        p(t^+)=S(\Delta m,p(t^-)).
    \end{equation} 
    
        
    \item Restart the classical dynamics \eqref{eq:fp-classical} again from the post-blowup profile obtained in \eqref{post-blowup-calculation-0}.
\end{enumerate}

We give two remarks before specifying how to obtain the blow-up size and the post-blowup profile.

First, this local resolution strategy naturally extends to a global-in-time description.
There can be a finite or infinite sequence of blow-up times $0<t_1<t_2<...<t_n<...$ such that
\begin{enumerate}
    \item At each blow-up time $t_i$, the blow-up is resolved according to the above procedure. In particular, we obtain the blow-up size and the post-blowup profile accordingly.
    \item During the \textit{inter-blowup intervals} $(0,t_1)\cup (t_1,t_2)...$, the solution solves the classical dynamics \eqref{eq:fp-classical}. In particular, the initial data at $t_i^+$ is the post-blowup profile obtained through \eqref{post-blowup-calculation-0}.
\end{enumerate}

Second, this discontinuous transition from $p(t^-)$ to $p(t^+)$ in the physical time $t$ can be realized as a continuous dynamic process in a stretched coordinate system. We introduce the \textit{dilated timescale} $\tau$, defined by the differential relation:
\begin{equation}\label{def-tau-sec2-1}
    d\tau= N(t)dt.
\end{equation} The blow-up time $t$ is mapped into a time interval $(\tau_a,\tau_b)$ in the new timescale, whose length is exactly the blow-up size $\tau_b-\tau_a=\Delta m$. By solving the dynamics in $\tau$ on that interval, we can obtain the post-blowup profile \eqref{post-blowup-calculation-0}. The details are given in  Section \ref{subsubsec:td-inter-a1=0}, see also Figure \ref{fig:time_dilation_mechanism}.

\begin{figure}[htbp]
    \centering
    \begin{tikzpicture}[>=Stealth, thick]

        \begin{scope}[local bounding box=globalView]
            \draw[->] (-0.2,0) -- (4.0,0) node[right] {$t$};
            \draw[->] (0,-0.2) -- (0,4.0) node[above] {$\tau$};
            
            \draw[blue!70!black, line width=1.2pt] (0,0) .. controls (1,0.5) and (1.5,1) .. (2,2) coordinate (jumpStart);
            
            \draw[red, line width=1.5pt] (2,2) -- (2,3.5) coordinate (jumpEnd);
            
            \draw[blue!70!black, line width=1.2pt] (2,3.5) .. controls (2.5,3.8) and (3.5,4) .. (4,4.2);
            
            \draw[dashed, thin, gray] (2,0) -- (jumpStart);
            \node[below] at (2,0) {$t^*$};
            
            \draw[decoration={brace,amplitude=4pt,raise=2pt},decorate, red] (2,2) -- (2,3.5) 
                node[midway, right=6pt, red] {$\Delta m$};
            
            \node[circle, draw=gray, dashed, minimum size=1.8cm, inner sep=0pt] (magnify) at (2,2.75) {};
        \end{scope}

        \begin{scope}[shift={(6.5,0.5)}, local bounding box=localView]
            
            \draw[gray, dashed, ultra thin] (magnify.north east) -- (-0.5, 3.2);
            \draw[gray, dashed, ultra thin] (magnify.south east) -- (-0.5, -0.5);

            \draw[->] (-0.5, 0) -- (4.5, 0) node[right] {$v$};
            \draw[->] (0, -0.2) -- (0, 3.0) node[above] {$\tilde{p}$};
            
            \draw[thick] (1.0, 2pt) -- (1.0, -2pt) node[below] {$V_R$};
            \draw[thick] (3.5, 2pt) -- (3.5, -2pt) node[below] {$V_F$};
            
            \fill[blue!10] (1.2, 0) .. controls (1.5, 2.0) and (2.8, 2.0) .. (3.5, 1.2) -- (3.5, 0) -- cycle;
            \draw[blue!80!black, thick] (1.2, 0) .. controls (1.5, 2.0) and (2.8, 2.0) .. (3.5, 1.2);
            
            \draw[->, blue!80!black, thick] (2.5, 1.0) -- (3.2, 1.0); 

            \draw[red, dashed, line width=1.2pt, ->] (3.5, 1.2) .. controls (3.8, 2.5) and (1.0, 2.5) .. (1.0, 1.8);
            \node[red, above, font=\scriptsize] at (2.25, 2.3) {Reset Flux};
            
            \draw[red, line width=1.5pt, ->] (1.0, 0) -- (1.0, 1.5);
            \node[red, right] at (1.0, 0.8) {$\delta$};

        \end{scope}

    \end{tikzpicture}
    \caption{Schematic of the time-dilation framework and the blow-up resolution mechanism. \textbf{Left:} The global time mapping, where the singularity at $t^*$ (infinite slope) is unfolded into a finite jump $\Delta m$ in the dilated time $\tau$. \textbf{Right:} The mesoscopic dynamics in the voltage domain $v$. Within the dilated interval, the probability density $\tilde{p}$ (blue) is transported across the threshold $V_F$. Crucially, at the onset of blow-up, the density at the boundary satisfies $\tilde{p}(V_F) > 0$. The mass flux exiting $V_F$ is instantaneously re-injected at the reset potential $V_R$ (red dashed path), accumulating as a Dirac delta function.}
    \label{fig:time_dilation_mechanism}
\end{figure}
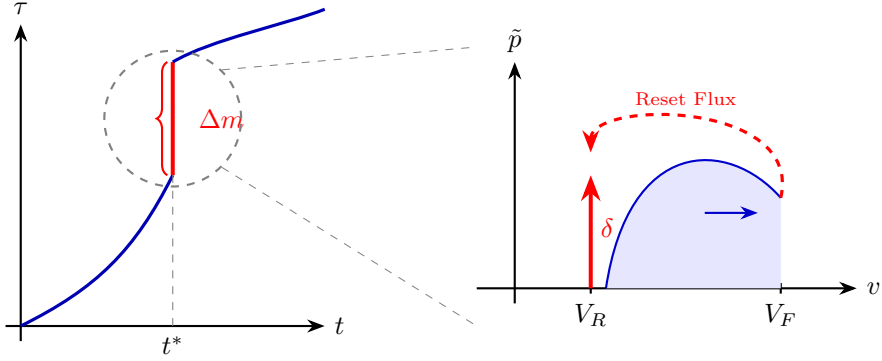

\subsection{Blow-up Size and Post-blowup Profile} 
\label{sec:postblowup}

Having established the general framework, we now specify the blow-up size $\Delta m$ and the post-blowup profile $p(t^+)$.
We introduce the following notation for the pre-blowup profile

\begin{equation}\label{def-p0-mu}
p_0(v):=p(t^-,v),\quad v\leq V_F,\qquad\quad\mu([a,b]):=\int_{a}^{b}p_0(v)dv,\quad -\infty\leq a<b\leq V_F.
\end{equation} 
 For simplicity, we always assume the pre-blowup profile is a probability density function. This is a natural assumption in view of the regularizing effect of \eqref{eq:fp-classical} thanks to the diffusion. Even if a Dirac mass at $V_R$ is generated in the post-blowup profile, it will be regularized instantly.




\paragraph{1. Blow-up size} Using the notation \eqref{def-p0-mu}, the blow-up size is given by
\begin{equation}\label{blowup-size-a1=0}
    \Delta m :=\inf \{x>0: \mu([V_F-bx,V_F])<x\}.
\end{equation}
Physically, this formula identifies the \textit{stopping point} of the cascade. The term $\mu([V_F-bx,V_F])$ represents the fraction of neurons potentially recruited by a collective voltage jump of size $bx$. The definition \eqref{blowup-size-a1=0} implies that the cascade proceeds as long as the recruited population is sufficient to sustain the voltage jump, and terminates at the smallest $x$ where the recruited mass falls short of $x$. This condition aligns with the cascade criteria established in probabilistic studies \cite{delarue2015particle} and computational neuroscience \cite{zhang2014distribution}.


\paragraph{2. Post-Blowup Profile}

The post-blowup profile consists of two terms
\begin{align}
        p(t^+)&=S(\Delta m,p(t^-))\\
              &=\Delta m\delta_{v=V_R}+S_{\Delta m}^rp(t^-),\label{post-blowup-calculation}
\end{align} which correspond to the proportion of neurons that spike at $t$ or not.
    \begin{enumerate}
        \item The first term $\Delta m\delta_{v=V_R}$ represents the reset of those neurons that spike at the blow-up time $t$. Here, an infinitesimal refractory period is taken into account.
        \item  The second term $S_{\Delta m}^rp(t^-)$ denotes how the remaining part of the population (which does not spike during the process) is redistributed by this event. We use the semigroup notation $S_{\Delta m}^r$ as it is indeed a linear semigroup operator indexed by the blow-up size $\Delta m>0$.
        
\end{enumerate}

We stress that in \eqref{post-blowup-calculation} an infinitesimal refractory period has been taken into account, leading to the Dirac mass at $V_R$. Physically, after spiking, a neuron enters a refractory period, during which it will not be influenced by the spikes of other neurons. An infinitesimal refractory period does not affect the PDE solution if the firing rate remains bounded. However, it makes a difference when $N$ blows up, as the blow-up means a fraction of the population spike during the same time $t=(t^-,t^+)$. The infinitesimal refractory period will hold those neurons at $V_R$ during the blow-up, leading to the formation of a Dirac mass at $V_R$.

When the firing neurons, constituting a total mass of $\Delta m$, are reset to $V_R$,
the remaining neurons, on the other hand, receive the total synaptic input $b\Delta m$ generated by the event, resulting in a uniform shift of their voltage profile.
Thus we define
\begin{equation}\label{blowup-operator-a1=0}
    (S_{\Delta m}^r p_0)(v):=p_0(v-b\Delta m),\qquad v\leq V_F.
\end{equation}
To summarize, the post-blowup profile is given by
\begin{equation}\label{post-blowup-full}
    p(t^+,v)=\Delta m\delta_{v=V_R}+p(t^-,v-b\Delta m),\qquad v\leq V_F.
\end{equation}
\begin{remark}
If the refractory period is not taken into account as in the previous works \cite{delarue2015particle,dou2022dilating,dou2024noisy}, the Dirac mass term in \eqref{post-blowup-full} will be replaced by a more regular term. For instance, in \cite{dou2024noisy} the corresponding post-blowup profile is 
\begin{equation}
        p(t^+,v)=\mathbb{I}(V_R\leq v\leq V_R+\Delta m)+p(t^-,v-b\Delta m),\qquad v\leq V_F.
\end{equation} The term $\mathbb{I}(V_R\leq v\leq V_R+\Delta m)$ is more regular, but less physical. 

The refractory period has been taken into accounts in the series of works \cite{taillefumier2022characterization,sadun2022global,papadopoulos2025physical} for a related but different model. In particular, in their models there is a firing-rate-dependent diffusion term $N(t)\p_{vv}p$. 
To our knowledge, we are the first to consider generalized solutions including the infinitesimal refractory period for the model \eqref{eq:fp-classical}. As our primary motivation is computation, we do not address the theoretical issues like the global existence of such generalized solutions, which has been addressed for less physical generalized solutions of the same model as in \cite{delarue2015particle,dou2024noisy}.

\end{remark}

We conclude this part by pointing out some properties of the post-blowup formulation \eqref{blowup-size-a1=0}-\eqref{post-blowup-full}. 

\begin{remark}[Direct properties of $\Delta m$]The definition \eqref{blowup-size-a1=0} implies 
\begin{equation}\label{jump-size-less-threshold}
    \mu([V_F-b x,V_F])\geq x,\qquad 0\leq x\leq \Delta m,
\end{equation} and also that for every $\epsilon>0$, there exists $ x\in (\Delta m, \Delta m+\epsilon)$ such that
\begin{equation}
        \mu([V_F-b x,V_F])< x.
\end{equation}
Moreover, {thanks to continuity} of $\mu$ we have an equality 
\begin{equation}\label{m-equality-bound}
   \mu([V_F-b\Delta m,V_F])=\Delta m,
\end{equation} which implies $0\leq \Delta m\leq 1$. 
\end{remark}
\begin{remark}[Mass Conservation] Thanks to \eqref{blowup-operator-a1=0} and \eqref{m-equality-bound} we derive
    \begin{align}\label{tmp-sec2-1}
                \int_{-\infty}^{V_F}(S_{\Delta m}^r p_0)(v)dv=\int_{-\infty}^{V_F-b\Delta m}p_0(v)dv=1-\mu([V_F-b\Delta m,V_F])=1-\Delta m.
    \end{align}
    Hence the post-blowup profile \eqref{post-blowup-calculation-0} is of mass one
    \begin{equation}
    \int_{-\infty}^{V_F}p(t^+,v)dv=   \int_{-\infty}^{V_F}\big((S^r_{\Delta m} p_0)(v)+\Delta m\delta_{v=V_R}\big)dv=1.
    \end{equation}
\end{remark}
\begin{remark}
    In terms of $S^r_{\Delta m}$ \eqref{blowup-operator-a1=0} and in view of \eqref{tmp-sec2-1}, we can rewrite \eqref{blowup-size-a1=0} as
    \begin{equation}
            \Delta m :=\inf \left \{x>0: 1-\int_{-\infty}^{V_F}(S^r_{x}p_0)(v)dv<x \right \}.
    \end{equation}
\end{remark}

\subsection{Microscopic/Mesoscopic Interpretations}

The definitions of the blow-up size $\Delta m$ and the post-blowup profile derived in Section \ref{sec:postblowup} are not arbitrary mathematical constructs; they are grounded in the physical reality of the neuronal network. In this section, we validate these definitions by interpreting them through three complementary lenses: the microscopic particle system (Section \ref{subsubsection:ps-inter}), the time-dilation dynamics (Section \ref{subsubsec:td-inter-a1=0}), and the probabilistic fragility criterion (Section \ref{subsubsection:fra-inter}).

These three approaches are unified under a multiscale perspective (Section \ref{subsubsec:rmk-a1=0}). While the macroscopic firing rate $N(t)$ diverges at the singularity, each of these perspectives resolves the blow-up by unfolding the dynamics onto a finer scale—whether it be the index of particles in a cascade, a stretched time coordinate, or the iteration of a mapping.

\subsubsection{Particle System Interpretation}\label{subsubsection:ps-inter}

We begin by establishing the microscopic foundation using a finite network of $N_p$ neurons. This particle model serves as the ground truth, from which the mean-field description emerges in the limit $N_p \to \infty$. The following formulation is adapted from the rigorous probabilistic frameworks in \cite{delarue2015particle} and the computational work \cite{zhang2014distribution}  (see also \cite{CiCP-30-820,DXZ2024}).

Consider a network of $N_p$ neurons, characterized by their membrane potentials $V^i_t$ ($1\leq i\leq N_p$, $t\geq0$). 
In the sub-threshold regime (where all $V_t^i < V_F$), the voltages evolve according to the coupled stochastic differential equations (SDEs)
\begin{equation}\label{SDE-OU}
    dV_t^i=-V_t^idt+\sqrt{2a}dB_t^i,\qquad 1\leq i\leq N_p,
\end{equation}
where $B_t^i$ are independent standard Brownian motions.
The interaction occurs strictly through the firing mechanism. When a neuron $V_t^i$ reaches the threshold $V_F$, it fires (spikes) and is instantaneously reset to $V_R$. Crucially, this spike delivers an excitatory kick of size $\alpha_{N_p}$ to all other neurons in the network
\begin{equation}\label{simple-spike}
    V_{t^-}^i=V_F\quad \Rightarrow \quad \begin{dcases}
        V_{t^+}^i=V_R,\\
        V_{t^+}^j=V_{t^-}^j+\alpha_{N_p},\quad j\neq i.
    \end{dcases}
\end{equation}
We adopt the standard mean-field scaling $\alpha_{N_p}=\frac{b}{N_p}$, which ensures the total interaction remains finite as the system size grows.

In the limit $N_p \to \infty$, this system converges to the solution $(p, N)$ of the Fokker-Planck equation \eqref{eq:fp-classical}-\eqref{def-N-classical} \cite{caceres2011analysis, delarue2015particle}, where the macroscopic firing rate $N(t)$ represents the normalized flux of spikes
\begin{equation}\label{formal-intuition}
    N(t){d}t = \lim_{N_p\rightarrow\infty} \frac{\text{Number of spikes during $[t,t+{d}t ]$}}{N_p}.
\end{equation}
In the classical regime, spike times do not concentrate in time, and $N(t)$ remains finite. However, the blow-up $N(t) \to \infty$ signifies a breakdown of the continuous description, corresponding physically to a \textit{multiple firing event (MFE)} or synchronization, where a macroscopic fraction of the population spikes within the same infinitesimal instant.

\paragraph{The Cascade Mechanism}
When a spike occurs (via \eqref{simple-spike}), the resulting voltage jump $\alpha_{N_p}$ may push other neurons across the threshold $V_F$, triggering further spikes and further jumps. This positive feedback loop creates an instantaneous avalanche.
To resolve this MFE, we describe the \textit{cascade dynamics} within the infinitesimal interval $(t^-, t^+)$. We drop the time index $t$ and denote the indices of the population as $[1 : N_p]:=\{1,2,...,N_p\}$.

\begin{itemize}
    \item \textbf{Initialization:} Given the pre-blowup configuration $\{V^1,...,V^{N_p}\}$, initialize the cascade counter $k=0$ and the set of fired neurons $\Gamma=\emptyset$.
    \item \textbf{Cascade Loop:} While there exists a neuron $i\in [1 : N_p]\backslash\Gamma$ such that $V^i\geq V_F$:
    \begin{enumerate}
        \item Identify all such neurons and denote their count by $\Delta k$.
        \item Add these indices $i$ to the fired set $\Gamma$.
        \item Apply the interaction: For all surviving neurons $j\in[1 : N_p]\backslash \Gamma$, update their voltage $V^j \leftarrow V^j+\alpha_{N_p}\Delta k$.
        \item Update the total spike count $k \leftarrow k+\Delta k$, and repeat.
    \end{enumerate}
    \item \textbf{Termination and Reset:} Once the cascade stabilizes (no new neurons cross $V_F$), set $V^i=V_R$ for all $i\in\Gamma$ (the refractory reset). The remaining neurons retain their shifted potentials.
\end{itemize}

This iterative procedure determines the set $\Gamma$ of synchronized neurons and the total MFE size $k = |\Gamma|$. We now explicitly connect this discrete mechanism to the macroscopic quantities defined in Section 2.2.

For the blow-up size, let $\mu_{N_p}$ be the empirical measure of the voltage configuration:
\begin{equation}\label{def-muN}
    \mu_{N_p}(A):=\int_{A}\left(\frac{1}{N_p}\sum_{i=1}^{N_p}\delta_{v=V^i}\right)dv,\qquad \text{for any interval $A\subseteq(-\infty,V_F]$}. 
\end{equation}
The following proposition establishes that the normalized cascade size $k/N_p$ is the discrete analog of the continuous blow-up size $\Delta m$ \eqref{blowup-size-a1=0}.

\begin{proposition}[See also {\cite[Proposition 3.4]{delarue2015particle} } or {\cite[Section 3.1]{zhang2014distribution}}]\label{prop:MFE}
    With $\alpha_{N_p}=\frac{b}{N_p}$, the MFE size $k$ resulting from the cascade process satisfies
\begin{equation}\label{particle-blowupsize-cond-2}
    \frac{k}{N_p}=\inf\left\{\frac{\ell}{N_p}:\,\ell\in[1..\,N_p],\quad\mu_{N_p}\left(\left[V_F-b\frac{\ell}{N_p},V_F\right]\right)<\frac{\ell+1}{N_p}\right\}.
\end{equation}
\end{proposition}
\begin{proof}
    Consider a hypothetical cascade of size $\ell$. The number of neurons that would be recruited by the resulting potential jump $b \frac{\ell}{N_p}$ is given by
\begin{equation}
    \# \{i:V^i\geq V_F-\alpha_{N_p}\ell\}.
\end{equation} 
For the cascade to reach size $\ell$ and \textit{continue}, the number of recruited neurons must exceed $\ell$. Conversely, the cascade \textit{stops} at size $k$ precisely when the system fails to recruit the $(k+1)$-th neuron. Thus, the final size $k$ is the smallest integer satisfying the stopping condition
\begin{equation}\label{particle-blowupsize-cond}
    k=\inf\{\ell\in[1:\,N_p]:\quad\# \{i:\,V^i\geq V_F-\alpha_{N_p}\ell\}<\ell+1\}.
\end{equation}
Rewriting the count in terms of the empirical measure $\mu_{N_p}$ yields \eqref{particle-blowupsize-cond-2}.
\end{proof} 

Finally, we examine the post-MFE profile. After the cascade terminates, the voltage configuration becomes:
 \begin{equation}
     V_{t^+}^i=\begin{dcases}
         V_R,\qquad &i\in\Gamma,\\
        V_{t^-}^i+{k}\alpha_{N_p},\qquad &i\in [1..\,N_p]\backslash\Gamma.
    \end{dcases}
 \end{equation} 
The first case ($i \in \Gamma$) corresponds to the Dirac mass at $V_R$ formed by the synchronized neurons. The second case corresponds to the surviving neurons, which have been shifted by $k \alpha_{N_p} = b (k/N_p)$. In the mean-field limit $N_p \to \infty$, where $k/N_p \to \Delta m$, this recovers exactly the structure of the generalized solution \eqref{post-blowup-full}: a Dirac mass of weight $\Delta m$ at $V_R$ plus the pre-blowup density translated by $b\Delta m$.

\subsubsection{Time-Dilation Interpretation}\label{subsubsec:td-inter-a1=0}

We now turn to the time-dilation perspective, which provides a continuous dynamical link between the pre- and post-blowup states. The central difficulty of the MFE is that the macroscopic clock $t$ effectively freezes relative to the fast cascading activity. To resolve this, we adopt the \textit{time-dilation} transformation, utilizing the network activity to rescale time.

A natural choice for the dilated timescale $\tau$ is defined by the relation:
\begin{equation}\label{dtau=Ndt}
    d\tau=N(t)dt.
\end{equation} 
Alternatively, one may define $d\tau=(N(t)+1)dt$ to avoid degeneracy when the firing rate vanishes ($N=0$). However, since we are strictly concerned with the blow-up regime where $N(t) \to \infty$, the two definitions are asymptotically equivalent. We adopt \eqref{dtau=Ndt} here for simplicity.

This transformation maps the instantaneous blow-up at $t=(t^-,t^+)$ into a mesoscopic interval $(\tau_1,\tau_2)$ in the dilated timescale. Without loss of generality, we set $\tau_1=0$ and denote the blow-up duration as $\Delta\tau:=\tau_2-\tau_1$.

The evolution of the probability density $\tilde{p}(\tau, v)$ during this interval is derived by analyzing the asymptotic limit of the Fokker-Planck equation \eqref{eq:fp-classical}. Dividing the equation by $N(t)$, we observe the following scaling behaviors as $N(t) \to \infty$:
\begin{itemize}
    \item The time derivative scales as $\frac{\p_t p}{N(t)} = \p_\tau \tilde{p}$.
    \item The finite drift and diffusion terms, $\frac{1}{N}\p_v(-vp)$ and $\frac{a}{N}\p_{vv}p$, are of order $O(1/N)$ and thus vanish in the limit.
    \item The interaction term scales as $\frac{1}{N}\p_v(bN p) = b\p_v \tilde{p}$.
    \item The reset term $N(t)\delta_{v=V_R}$ is formally of order $O(N)$, but we explicitly discard it to model the \textit{infinitesimal refractory period}. Physically, this means neurons that fire during the cascade are removed from the domain and do not re-enter until the event concludes.
\end{itemize}
Consequently, the dynamics in the dilated interval are governed by a pure transport equation:
\begin{equation}\label{tau-blowup-a1=0}
    \begin{aligned}
        \p_{\tau}\tilde{p}+b\p_v\tilde{p}&=0,\quad \tau\in(0,\Delta\tau),\, v<V_F,\\
    \tilde{p}(\tau=0,v)&=p_0(v),\quad v\leq V_F.
\end{aligned}
\end{equation} 
To determine the duration of the blow-up $\Delta\tau$, we introduce an auxiliary variable $M(\tau)$ evolving according to:
\begin{equation}\label{dyn-M-a1=0}
\begin{aligned}
        \frac{d}{d\tau}M(\tau)&=b\tilde{p}(\tau,V_F)-1,\qquad\tau\in(0,\Delta\tau).\\\quad M(0)&=0.
\end{aligned}
\end{equation}
The blow-up interval is defined as the time required for this auxiliary variable to drop below zero:
\begin{equation}\label{def-Deltatau-a1=0}
    \Delta \tau:=\inf\{\tau>0: M(\tau)<0\}.
\end{equation} 
The following proposition confirms that this dynamical stopping criterion is equivalent to the condition \eqref{blowup-size-a1=0} in Section \ref{sec:postblowup}. 

\begin{proposition}\label{Prop:time-dilation-jump-a1=0}
 The length of the blow-up interval $\Delta\tau$ defined in \eqref{def-Deltatau-a1=0} satisfies
    \begin{equation}
\begin{aligned}
            \Delta \tau=\inf \{x>0: \mu([V_F-bx,V_F])<x\}.
\end{aligned}
    \end{equation}
\end{proposition} 
\begin{proof}
   The expression of $\tilde{p}$ can be solved explicitly in \eqref{tau-blowup-a1=0}  as a translation of $p_0$ with the speed $b>0$
\begin{equation}\label{tmp-translation}
       \tilde{p}(\tau,v)=p_0(v-b\tau),\quad \tau>0,V\leq V_F.
   \end{equation} We can also obtain an expression of $M(\tau)$ by simply integrating the ODE \eqref{dyn-M-a1=0} \begin{equation}\label{expresion-M-a1zero}
       M(\tau)=b\int_0^{\tau}\tilde{p}(\tilde{\tau},V_F)d\tilde{\tau}-\tau.
   \end{equation} Using \eqref{tmp-translation} we derive by a change of variable
   \begin{align*}
b\int_0^{\tau}\tilde{p}(\tilde{\tau},V_F)d\tilde{\tau}&=b\int_0^{\tau}p_0(V_F-b\tilde{\tau})d\tilde{\tau}\\&=\int_0^{\tau}p_0(V_F-b\tilde{\tau})d(b\tilde{\tau})=\int_{V_F-b\tau}^{V_F}p_0(v)dv=\mu([V_F-b\tau,V_F]).
   \end{align*} Hence, together with \eqref{expresion-M-a1zero} we conclude that $M(\tau)<0$ is equivalent to
  $
       \mu([V_F-b\tau,V_F])<\tau,$
    which implies the desired result.
   \end{proof}

Upon exiting the dilated phase, the post-blowup profile is constructed as:
\begin{equation}\label{a1z-pobl-p}
    p(t^+,v)=\tilde{p}(\Delta\tau,v)+\Delta\tau\delta_{v=V_R}.
\end{equation}
In this way, we recover the previous definitions of the post-blowup profile \eqref{post-blowup-calculation}, \eqref{blowup-size-a1=0}.

\begin{remark}[Bound on Blow-up Size]
    Prop. \ref{Prop:time-dilation-jump-a1=0} implies that $\Delta\tau \le 1$. This upper bound guarantees that the dynamics in the dilated timescale will not persist indefinitely; the system is guaranteed to exit the $\tau$-phase and return to the physical timescale $t$.
\end{remark}

\begin{remark}[Physical Interpretation of $M(\tau)$]
The variable $M(\tau)$ acts as a \textit{backlog counter}. It represents the accumulated ``virtual mass'' of neurons that have crossed the threshold $V_F$ but have not yet been effectively processed (reset).
The term $b\tilde{p}(\tau,V_F)$ represents the incoming flux of neurons driven across the threshold, while the term $-1$ represents the normalized rate at which the system processes these spikes in the dilated frame.
Crucially, the cascade persists as long as there is a backlog ($M(\tau) \ge 0$), even if the instantaneous incoming flux drops below the processing rate (i.e., even if $M'(\tau) < 0$). The event terminates only when this accumulated backlog is fully depleted.
\end{remark}
\begin{remark}[Mathematical Justification]
    The dynamics of $M(\tau)$ can be rigorously derived in \cite{dou2024noisy} as the asymptotic limit of a regularized model, where the equation is extended to the whole line and $M(\tau)$ is realized as the mass of neurons with voltages in $[V_F,+\infty)$. Yet the refractory state is not taken into account in \cite{dou2024noisy}.
\end{remark}

\subsubsection{Fragility Interpretation}\label{subsubsection:fra-inter}

Finally, we examine the blow-up from the perspective of \textit{structural stability}, utilizing the concept of ``fragility'' introduced in \cite{hambly2019mckean}. This approach characterizes the sensitivity of the system to external perturbations and provides an iterative viewpoint on the synchronization size.

Given a coupling strength $b>0$, a pre-blowup probability distribution $\mu$ on $(-\infty,V_F]$, and an external stimulus of size $\eps>0$, we define the recursive sequence:
\begin{align}
    f_0(\mu,\eps)&:=\mu([V_F-\epsilon,V_F]),\\
    f_{n+1}(\mu,\eps)&:=\mu([V_F-\epsilon-bf_n(\mu,\epsilon),V_F]).
\end{align} 
Physically, the sequence $f_n(\mu, \eps)$ iteratively calculates the cumulative recruitment of neurons. Starting with an initial stimulus $\epsilon$, each step accounts for the additional neurons recruited by the voltage jump induced by the previous wave. Since $f_n$ is non-decreasing, the limit exists:
\begin{equation}
    f_{\infty}(\mu,\eps):=\lim_{n\rightarrow+\infty} f_{n}(\mu,\eps).
\end{equation} 
This limit represents the total synchronization size resulting from the stimulus $\eps$. The connection to the blow-up size $\Delta m$ is established by examining the limit of vanishing stimulus:

\begin{proposition}[{\cite[Proposition 2.4]{hambly2019mckean}}]\label{prop:fragility}
    Given $b>0$, for an atomless probability measure $\mu$ on $(-\infty,V_F]$, we have
    \begin{equation}\label{fragility-condition}
        \lim_{\eps\rightarrow0^+}f_{\infty}(\mu,\eps)=\inf \{x>0: \mu([V_F-bx,V_F])<x\}.
\end{equation}
\end{proposition}
The right-hand side of \eqref{fragility-condition} coincides exactly with our definition of the synchronization size $\Delta m$ in \eqref{blowup-size-a1=0}. 

This result highlights a critical instability. For a stable distribution, $\lim_{\eps\to 0} f_\infty = 0$. However, for a \textit{fragile} distribution, the double limits do not commute:
$$\lim_{n\rightarrow+\infty}\left(\lim_{\epsilon\rightarrow0^+}f_n(\mu,\eps)\right)=0 \quad \neq \quad \lim_{\epsilon\rightarrow0^+}\left(\lim_{n\rightarrow+\infty}f_n(\mu,\eps)\right) = \Delta m > 0.$$
This non-commutativity implies a physical discontinuity: a system is fragile if an infinitesimal stimulus triggers a macroscopic avalanche. The mathematical blow-up size is thus identified as the inherent susceptibility of the network state.

Note that this fragility approach  is originally proposed for models in finance \cite{hambly2019mckean}, which does not have the reset mechanism. Therefore, we do not discuss post-blowup profile here.

\subsubsection{Summarizing Remarks}\label{subsubsec:rmk-a1=0}

From a theoretical standpoint, the three perspectives discussed above—the particle cascade, the time-dilation dynamics, and the fragility analysis—yield a unified description of multiple firing events. They all consistently identify the same synchronization size \eqref{blowup-size-a1=0}. Moreover, the particle system and the time-dilation approach gives the same post-blowup profile \eqref{post-blowup-full}.


However, from a computational perspective, the \textit{Time Dilation} approach distinguishes itself by its algorithmic utility. Unlike the particle interpretation, which relies on discrete enumeration, or the fragility approach, which requires iterative limit procedures, the time-dilation framework transforms the instantaneous singularity into a continuous dynamical process. This allows us to {seamlessly} capture the blow-up size and the reorganized voltage distribution by simply integrating the transport equation \eqref{tau-blowup-a1=0} and the auxiliary variable \eqref{dyn-M-a1=0} forward in the dilated time $\tau$. Consequently, this dynamical formulation provides the natural blueprint for the numerical algorithms we will develop in Section \ref{sec:algorithm}, enabling a solver that automatically navigates through singularities without ad-hoc rules. 

\begin{remark}[Connection between $\tau$ and Cascade Steps]
    We highlight a direct connection between the dilated timescale $\tau$ and the cascade mechanism of the particle system.
    Recall that the dilated time is defined by $d\tau=N(t)dt$. In the context of the particle system (see \eqref{formal-intuition}), the firing rate $N(t)$ formally represents the density of spikes.
    Consequently, the increment $d\tau$ corresponds to the normalized number of spikes:
    \begin{equation}
        \Delta \tau \approx \frac{1}{N_p} \times (\text{number of spikes in the cascade}).
    \end{equation}
    Therefore, the computational time variable $\tau$ is not merely an artificial regularizer; it effectively tracks the progress of the physical avalanche. Indeed, the term ``cascade time-axis'' is used in \cite{delarue2015particle} to describe the index in the cascade loop of the particle systems.  The consistency ensures that our continuum solver, while efficient, remains faithful to the microscopic causality of the cascade process.  
\end{remark}

\section{Numerical Algorithms}\label{sec:algorithm}

Building upon the theoretical framework established in Section \ref{sec:theory}, this section develops a robust numerical strategy to resolve the mean-field Fokker-Planck equation \eqref{eq:fp-classical} beyond the onset of firing rate blow-ups. The central computational challenge lies in the emergence of MFEs, where the standard macroscopic description fails due to the singular nature of the firing rate $N(t)$. To bridge the gap between analytical theory and high-fidelity simulation, we propose a scheme that adaptively captures these discrete synchronization events while maintaining the physical consistency of the underlying neuronal population dynamics.

From a multiscale perspective, the time-dilation approach allows us to treat the blow-up not as a numerical failure, but as a transition to a \textit{mesoscopic} regime. While the \textit{macroscopic} PDE in natural time $t$ provides an efficient global description, the dilated timescale $\tau$ functions as a mesoscopic lens that resolves the fine-grained ``hidden'' dynamics during a cascade. Unlike hybrid methods that necessitate a complete switch to \textit{microscopic} particle systems---often incurring significant computational overhead---our approach remains within the continuum framework, using the $\tau$ timescale as a localized regularizer to resolve singularities. 

The general procedure of the algorithm is as follows:
\begin{enumerate}
    \item \textbf{Macroscopic Evolution:} While the firing rate remains bounded \textit{numerically}, we solve {numerically} the classical Fokker-Planck equation \eqref{eq:fp-classical} in the natural time $t$. 
    
    \item \textbf{Mesoscopic Singularity Resolution:} Upon detecting \textit{numerically} a blow-up of the firing rate, the solver transitions to the dilated timescale $\tau$. In this regime, the instantaneous event is unfolded into a continuous process governed by \eqref{tau-blowup-a1=0}--\eqref{dyn-M-a1=0}. Solving this process numerically allows us to resolve the cascade and determine the blow-up size $\Delta m$ as well as the post-blowup profile.
    
    \item \textbf{Restart the Macroscopic Evolution:} Following the resolution of the MFE, we obtain the post-blowup profile which contains a Dirac mass at $V_R$ due to the reset.  With the post-blowup profile, we restart the numerical evolution of the classical dynamics \eqref{eq:fp-classical} in $t$ timescale.    
\end{enumerate}


The transition from the analytical framework in Section \ref{sec:theory} to a practical numerical solver entails three major computational challenges:

\begin{itemize}
    \item \textbf{Stability in the Convection-Dominated Regime:} As the system approaches a blow-up time, the firing rate $N(t)$ is large and the drift term in the Fokker-Planck equation becomes increasingly dominant, which is a numerical challenge shared by convection-diffusion problems (e.g. \cite{stynes2018convection}). To ensure an accurate approximation, we need to investigate whether the structure-preserving framework established in \cite{hu2021structure,he2022structure} remains robust across varying firing rates.
    
    \item \textbf{Mesh-Independent Blow-up Detection:} A critical requirement for the algorithm is a reliable trigger for timescale switching. We need to numerically detect the blow-up of $N(t)$. Unlike heuristic thresholds in \cite{DXZ2024}, we derive a natural criterion based on the loss of Dirichlet boundary condition. It can be recast in a mesh-independent form \eqref{numerical-dbcl} and is consistent with recent theoretical result \cite{munoz2025free}. Thus it provides a rigorous numerical signature of the loss of boundary control, signaling the onset of an MFE in a manner that remains consistent as the mesh is refined. The criteria is also consistent with other parts of the algorithms which ensures seamless switching between scales.
    
    \item \textbf{Regularization of the Post-MFE Singularity:} The immediate consequence of a multiple firing event is the accumulation of a Dirac mass at the reset potential $V_R$. Since fixed-grid methods cannot resolve such singularities without excessive diffusion, we adopt a semi-analytical approach. This hybrid treatment allows the singular component to evolve according to its local linearized dynamics until it is sufficiently regularized for seamless reintegration into the continuous density.
\end{itemize}

In the following we detail our algorithms. Along the presentation we show theoretical properties which help us address the numerical challenges discussed above. A compact summary of the algorithm is given at the end of this section. 

\subsection{Mesh Setting}\label{subsec:mesh}

We truncate the spatial\footnote{More precisely, it is the `voltage domain'. } domain from $(-\infty,V_F]$ to $[V_{\min},V_F]$ with $V_{\min}<V_R$ and $V_R-V_{\min}$ reasonably large such that the relevant dynamics mostly happen in $[V_{\min},V_F]$.  We introduce a spatially uniform mesh with gridsize $h>0$ and denote $v_i:=V_{\min}+ih$ for $i=0,1,...n$. Here we choose $h>0$ such that $n=\frac{V_F-V_{\min}}{h}$ is an integer and thus $v_n=V_F$. We also set $V_R$ to be a grid point (by choosing $(h, V_{\min})$ properly), that is, there exists an integer $0<i_R<n$ such that
\begin{equation}\label{def:ir}
    V_R=V_{\min}+i_Rh.
\end{equation}

We solve the dynamics of the classical solution in $t$ timescale and only switch to the dilated timescale $\tau$ to resolve the blow-up. The timesteps in each timescale are denoted by $\delta t$ and $\delta \tau$, respectively.

\subsection{Before Blow-up: Classical Dynamics}\label{subsubsec:cl-a1z}

We first present the numerical schemes for the classical PDE \eqref{eq:fp-classical}, when the firing rate does not blow-up. We adopt the schemes in \cite{hu2021structure,he2022structure} which satisfy several structure-preserving properties. First, we consider a {Chang-Cooper type reformulation} \cite{ChangCooper1970} ({aka} the Scharfetter-Gummel reformulation \cite{ScharfetterGummel1969}) to rewrite \eqref{eq:fp-classical} as
\begin{equation}\label{eq:SG-reformulate}
    \p_tp=\p_v(F(t,v)),\qquad t>0,v<V_F,
\end{equation} where the flux $F(t,v)$ is written as
\begin{align}
F(t,v):&=a\p_vp-(-v+bN(t))p+N(t)H(v-V_R)\\&= a\mathcal{M}(t,v)\p_v(p(t,v)/\mathcal{M}(t,v))+N(t)H(v-V_R).\label{def:F}
\end{align} Here $H(v):=\mathbb{I}_{v>0}$ is the Heaviside function and $\mathcal{M}(t,v)$ is given by
\begin{equation}\label{def:M}
    \mathcal{M}(t,v):=\exp\left(\frac{1}{a}\left(-\frac{1}{2}v^2+{bN(t)}v\right)\right).
\end{equation} 

Now we consider a finite volume discretization based on the reformulation \eqref{eq:SG-reformulate}-\eqref{def:F}-\eqref{def:M}. We first consider the semi-discretization: discretizing the space (voltage) variable while keeping the time variable continuous. For $p_i(t)\approx p(t,v_i)$, we have
\begin{equation}\label{semi-discrete}
    \p_t p_i =\frac{F_{i+\frac{1}{2}}-F_{i-\frac{1}{2}}}{h},\qquad t>0, i=1,...,n-1,
\end{equation} where the flux $F_{i+\frac{1}{2}}$ at the half-grid is defined as
\begin{equation}\label{def:F-discrete}
    F_{i+\frac{1}{2}}:=a\mathcal{M}_{i+\frac{1}{2}}\frac{\frac{p_{i+1}}{\mathcal{M}_{i+1}}-\frac{p_{i}}{\mathcal{M}_{i}}}{h}+N_h(t)\mathbb{I}_{i\geq i_R},\qquad t>0,i=0,1,...n-1.
.\end{equation} Here we recall $i_R$ is defined in \eqref{def:ir}. For $\mathcal{M}_i$ and $\mathcal{M}_{i+\frac{1}{2}}$, we simply choose
\begin{equation}\label{def-Mi}
\mathcal{M}_i(t):=\mathcal{M}_h(t,V_{\min}+ih),\quad \mathcal{M}_{i+\frac{1}{2}}(t):=\mathcal{M}_h\left(t,V_{\min}+(i+\frac{1}{2})h\right),
\end{equation} where the function $\mathcal{M}_h(t,v)$ is the analog of the function $\mathcal{M}(t,v)$ in \eqref{def:M} with $N_h(t)$ in place of $N(t)$
\begin{equation}\label{def-Mh}
    \mathcal{M}_h(t,v):=\exp\left(\frac{1}{a}\left(-\frac{1}{2}v^2+{bN_h(t)}v\right)\right).
\end{equation} The numerical firing rate $N_h$ is defined as a first order discretization
\begin{equation}\label{def:Nh}
    N_h(t):=a\frac{p_{n-1}}{h}\approx -a\p_vp(t,V_F).
\end{equation} In principle, $p_i,F_{i+\frac{1}{2}}$ and $\mathcal{M}_i$ also depend on the gridsize $h$, but we omit the notational dependence here for simplicity.

For the boundary conditions, we set
\begin{equation}
    F_{\frac{1}{2}}=F_{n-\frac{1}{2}}=0.
\end{equation} Note that in view of \eqref{def:F-discrete}, this discrete boundary condition determines $p_0$ and $p_{n}$ implicitly, which shall be close to zero (but might not be exactly zero at the numerical level).

The semi-discrete scheme \eqref{semi-discrete} can be written more explicitly as
\begin{align}
    \p_tp_i=\frac{a}{h^2}\left(\frac{\mathcal{M}_{i+1/2}}{\mathcal{M}_{i+1}}p_{i+1}-\frac{\mathcal{M}_{i+1/2}+\mathcal{M}_{i-1/2}}{\mathcal{M}_i}p_i+\frac{\mathcal{M}_{i-1/2}}{\mathcal{M}_{i-1}}p_{i-1}\right)+\frac{1}{h}N_h(t)\mathbb{I}_{i=i_R},\qquad t>0,2\leq i\leq n-2,
\end{align} with the equations for $i=1$ and $i=n-1$ 
\begin{align}
    \p_tp_{1}&=\frac{a}{h^2}\left(\frac{\mathcal{M}_{1+1/2}}{\mathcal{M}_{2}}p_{2}-\frac{\mathcal{M}_{1+1/2}}{\mathcal{M}_1}p_1\right),\qquad t>0,\\
    \p_tp_{n-1}&=\frac{a}{h^2}\left(-\frac{\mathcal{M}_{n-1-1/2}}{\mathcal{M}_{n-1}}p_{n-1}+\frac{\mathcal{M}_{n-1-1/2}}{\mathcal{M}_{n-1-1}}p_{n-2}\right)-\frac{1}{h}N_h(t),\qquad t>0.
\end{align} We introduce the short-hand notations
\begin{align}\label{def:r+-r-}
    r_{i,+}:=\frac{\mathcal{M}_{i+1/2}}{\mathcal{M}_{i}},\qquad r_{i,-}:=\frac{\mathcal{M}_{i-1/2}}{\mathcal{M}_{i}}.
\end{align} And then the equation for can be written more compactly as
\begin{equation}\label{compact-r}
    \p_tp_i=\frac{a}{h^2}\left(r_{i+1,-}p_{i+1}-(r_{i,+}+r_{i,-})p_i+r_{i-1,+}p_{i-1}\right),\qquad t>0,i=2,...,n-2,
\end{equation} and similarly for $i=1,n-1$.

\paragraph{Full Discretization}

Now we include the temporal discretization to obtain a fully discrete scheme. For $p_i^m\approx p(m\delta t,V_{\min}+ih)$, we consider
\begin{equation}\label{full-dis}
    \frac{p^{m+1}_i-p_i^m}{\delta t}=\frac{F^{m^*}_{i+1/2}-F_{i-1/2}^{m^*}}{h}.
\end{equation} At the fully discrete level, we define the firing rate similarly to \eqref{def:Nh} as follows
\begin{equation}\label{def-Nhm}
    N_h^{m}:=a\frac{p_{n-1}^m}{h}.
\end{equation} The key is to choose a proper flux $F^{m*}$, which has dependencies on $N$ and $p$  \eqref{def:F}. It depends on the firing rate $N(t)$ in two ways: a nonlinear dependence through $\mathcal{M}(t,v)$, and an explicit dependence in the second term $N(t)H(v-V_R)$. At the discrete level, a key choice is on using $(N,p)$ from which step (explicitly from step $m$ or implicitly from step $m+1$) in the definition of $F^{m^*}$, resulting a variety of schemes from the fully explicit to the fully implicit. 

In this work, we consider the following choice as in \cite{hu2021structure}. For the (linear) dependence on $p$, we use the data at step $m+1$, treating it implicitly. For the nonlinear dependence on $N$ through $M$ (due to interaction), we use the data at step $m$, treating it explicitly. For the linear dependence on $N$ (due to reset), we also use the data at step $m$ to treat it explicitly, leading to
\begin{equation}\label{F-choice-1}
F^{m^*}_{i+1/2}:=a\mathcal{M}_{i+\frac{1}{2}}^m\frac{\frac{p_{i+1}^{m+1}}{\mathcal{M}_{i+1}^m}-\frac{p_{i}^{m+1}}{\mathcal{M}_{i}^m}}{h}+N_h^m\mathbb{I}_{i\geq i_R},\qquad i=0,1,...n-1.
\end{equation} Here $\mathcal{M}_{i}^m$ is defined similarly to \eqref{def-Mi}-\eqref{def-Mh} with the fully discrete firing rate $N_h^m$ in place of the semi-discrete firing rate $N_h(t)$. This choice of flux \eqref{F-choice-1} leads to an implicit but linear equation to solve at each step. 

To gain a better understanding of the scheme, we write down the linear equation involved at each step explicitly. Set $\mathbf{p}^m:=(p_1^m,...,p_{n-1}^m)^T$ to be the $n-1$ dimensional vector for the numerical solution. The linear equation corresponding to \eqref{F-choice-1} is
\begin{equation}\label{linear-choice-1}
A\mathbf{p}^{m+1}=\mathbf{p}^m+\frac{\delta t}{h}N_h^m(\mathbf{e}_{i_R}-\mathbf{e}_{n-1}),
\end{equation} where $\mathbf{e}^i$ is the indicator vector defined as $(\mathbf{e}^i)_j:=\mathbb{I}_{i=j}$ for $i,j=1,...,n-1$, and $A=(a_{i,j})_{(n-1)\times (n-1)}$ is a tri-diagonal matrix determined by
\begin{align}
    a_{i,i+1}&:=-a\frac{\delta t}{h^2}r_{i+1,+}^m,\qquad i=1,2,...,n-1,\\
    a_{i,i}&:=1+a\frac{\delta t}{h^2}(r_{i,+}^m\mathbb{I}_{i\leq n-2}+r_{i,-}^m\mathbb{I}_{i\geq 2}),\qquad i=1,2,...,n-1,\\
    a_{i,i-1}&:=-a\frac{\delta t}{h^2}r_{i-1,+}^m,\qquad i=2,...,n-1,
\end{align} where $r_{i,+}^m$ and $r_{i,-}^m$ are defined similarly to \eqref{def:r+-r-}
\begin{align}\label{def:r+-r-full}
    r_{i,+}^m:=\frac{\mathcal{M}^m_{i+1/2}}{\mathcal{M}_{i}^m},\qquad r_{i,-}^m:=\frac{\mathcal{M}_{i-1/2}^m}{\mathcal{M}_{i}^m}.
\end{align} We note that the matrix $A$ itself depends on the step $m$. 
\begin{remark}
There is an alternative choice of the flux proposed in \cite{he2022structure}, which treats the the reset firing rate implicitly using the data at step $m+1$. 
 It is given by
\begin{equation}\label{F-choice-2}
F^{m^*}_{i+1/2}:=a\mathcal{M}_{i+\frac{1}{2}}^m\frac{\frac{p_{i+1}^{m+1}}{\mathcal{M}_{i+1}^m}-\frac{p_{i}^{m+1}}{\mathcal{M}_{i}^m}}{h}+N_h^{m+1}\mathbb{I}_{i\geq i_R},\qquad i=0,1,...n-1.
\end{equation} The only difference compared to \eqref{F-choice-1} is in the last term: using $N_h^{m+1}\mathbb{I}_{i\geq i_R}$ instead of $N_h^{m}\mathbb{I}_{i\geq i_R}$.  As $N_h^{m+1}$ depends on $p^{m+1}$ linearly \eqref{def-Nhm}, this alternative choice also leads to a linear equation to solve at each step, written as
    \begin{equation}
A\mathbf{p}^{m+1}=\mathbf{p}^m+a\frac{\delta t}{h^2}p^{m+1}_{n-1}(\mathbf{e}_{i_R}-\mathbf{e}_{n-1}),
\end{equation} where the matrix $A$ is the same as in \eqref{linear-choice-1}. We can move the second term from the right hand side to the left, leading to a linear equation of the more standard form
\begin{equation}
\tilde{A}\mathbf{p}^{m+1}=\mathbf{p}^{m},
\end{equation} where $\tilde{A}$ is `almost' a trigonal matrix: it has one additional non-zero element. The second choice can be more stable as the reset firing rate is treated implicitly. However, in this work we focus on the first choice \eqref{F-choice-1} as it is simpler and is enough for our purposes.
 \end{remark}

\subsection{Detect Blow-up: Criteria}

The blow-up of the firing rate $N(t)$ manifests mathematically as a boundary singularity $-\p_vp(t,V_F)=+\infty$. To resolve this numerically, we must establish a discrete threshold that distinguishes high-activity regimes from actual singularities. A natural quantitative starting point is to monitor the numerical firing rate $N_h^m$. Our first detection threshold given as follows
\begin{equation}\label{cr-blow-up-a0=0}
    N_h^m\geq \frac{a}{bh}.
\end{equation} Intuitively, this means that the macroscopic description breaks down when the firing rate scales inversely with the grid size, i.e., $N_h \sim O(1/h)$.

However, a deeper understanding is achieved by analyzing the structural implications of this condition. By substituting the definition of $N_h^m$ \eqref{def-Nhm} into \eqref{cr-blow-up-a0=0}, we derive an equivalent condition on the boundary density
\begin{equation}\label{numerical-dbcl}
   b p^m_{n-1}\geq 1.
\end{equation}
This condition \eqref{numerical-dbcl} reveals the physical essence of the blow-up: it marks the breakdown of the Dirichlet boundary condition. The value $1/b$ is hinted in the condition \eqref{blowup-size-a1=0} as $p(V_F)=\frac{1}{b}$ is the threshold to have a non-trivial blow-up size. Furthermore, this threshold aligns with recent theoretical findings for related models \cite{munoz2025free}, where $1/b$ is identified as a critical barrier for the boundary density to have a blow-up. Unlike the heuristic rate threshold \cite{DXZ2024}, the density threshold \eqref{numerical-dbcl} avoids ad-hoc parameter and provides a {mesh-independent} signature of the onset of synchronization.

This criterion also ensures the logical consistency of the algorithmic cycle. As will be shown in Prop. \ref{prop:post-blowup}, the numerical post-blowup profile constructed by our method naturally satisfies $b p_{n-1} < 1$. Consequently, the use of \eqref{numerical-dbcl} as a trigger prevents the spurious immediate re-triggering of the blow-up procedure, ensuring a clean exit from the $\tau$-dynamics back to the $t$-dynamics.

Crucially, this detection criterion serves as a stability guarantee for the pre-blowup scheme. In the regime immediately preceding a synchronization event, the system becomes strongly drift-dominated ($N(t) \gg 1$), potentially creating sharp boundary layers of width $O(1/N)$ that challenge standard discretizations \cite{stynes2018convection}. However, our criterion acts as a safeguard. As long as the system remains below the threshold (i.e., Eq. \eqref{numerical-dbcl} is not met), the exponential fitting coefficients in our scheme remain uniformly bounded, preventing numerical breakdown even in the presence of large drifts. This robustness is formally established in the following proposition.

\begin{proposition}\label{prop:bounded}Suppose  \eqref{cr-blow-up-a0=0} does not hold and $N_h^m$ as defined in \eqref{def-Nhm} is non-negative, i.e.,
\begin{equation}\label{bdnh}
   0\leq  N_h^m<\frac{a}{bh}.
\end{equation} The the coefficients $r_{i,+}^m,r_{i,-}^m$ defined in \eqref{def:r+-r-full} satisfies
\begin{equation}
   e^{-1/2}(1+O(h))\leq r_{i,+}\leq 1+O(h),\qquad    1+O(h)\leq r_{i,-}\leq e^{1/2}(1+O(h)).
\end{equation}
    \end{proposition} 
The non-negativity of $N_h^m$ is natural, as by definition \eqref{def-Nhm} it shall follow from the non-negativity of $p_i^m$. Positivity-keeping properties of the scheme \eqref{full-dis}-\eqref{F-choice-1} (with a condition related to the CFL number) have been proved in \cite[Theorem 2.2]{hu2021structure}.


\begin{proof}[Proof of Proposition~\ref{prop:bounded}]
    We explicitly analyze the exponent in the definition of $r_{i,+}^m$. Recalling the definitions of $\mathcal{M}_h$ \eqref{def-Mh} and the ratio \eqref{def:r+-r-full}, we have
    \begin{align}
        r_{i,+}^m = \frac{\mathcal{M}^m_{i+1/2}}{\mathcal{M}_{i}^m} &= \exp\left( \frac{1}{a} \left[ -\frac{1}{2}(v_{i+1/2}^2 - v_i^2) + bN_h^m (v_{i+1/2} - v_i) \right] \right).
    \end{align}
    Using the grid relations $v_{i+1/2} - v_i = h/2$ and $v_{i+1/2}^2 - v_i^2 = v_i h + O(h^2)$, the exponent becomes
    \begin{equation}
        \frac{1}{a} \left[ -\frac{1}{2}v_i h + \frac{1}{2} b N_h^m h + O(h^2) \right].
    \end{equation}
    Since the computation domain is bounded, the term $-\frac{1}{2a}v_i h$ is simply $O(h)$. The critical interaction term involves the product $N_h^m h$.  The condition \eqref{bdnh} implies that $0\leq \frac{b}{2a}N_h^m h<\frac{1}{2}$, which allows us to conclude for $r_{i,+}^m$. A similar derivation applies to $r_{i,-}^m$, completing the proof.

\end{proof}

\subsection{During Blow-up: Time dilation}\label{subsubsec: dilation-blowup}

Once the switching criterion \eqref{numerical-dbcl} is satisfied, the solver suspends the macroscopic evolution in $t$ and activates the mesoscopic dynamics in the dilated timescale $\tau$. The primary objective in this regime is to resolve the instantaneous reorganization of the voltage distribution and determining the blow-up size $\Delta m$.

To discretize the mesoscopic system \eqref{tau-blowup-a1=0}--\eqref{dyn-M-a1=0}, we retain the same spatial grid with mesh size $h$ and select the time step $\delta\tau$ to satisfy the unit CFL condition
\begin{equation}\label{tau-timestep}
    \delta\tau := \frac{1}{b}h.
\end{equation}
This choice eliminates numerical diffusion in the transport step, as the upwind scheme becomes an exact shift operator. We denote the discrete variables in this regime by $\tilde{p}_i^k \approx \tilde{p}(k\delta\tau, v_i)$ and $M^k \approx M(k\delta\tau)$, where the index $k$ counts the steps in the dilated timescale, distinct from the macroscopic step $m$.

{ We initialize those variables from the numerical pre-blowup profile $p^m_i$ in the last step of the classical dynamics, as follows
\begin{equation}\label{initialize-tau}
    \tilde{p}^0_i:=p_i^m,\quad i=1,...,n-1,\qquad M^0:=0,\qquad \tau^0:=0,
\end{equation} where we also explicitly records the time $\tau$ in the dilation timescale.}

The dynamics for $\tilde{p}$ \eqref{tau-blowup-a1=0} is a simple transport equation and we use an explicit upwind scheme to discretize it
\begin{equation}\label{ic-n-a1z}
    \frac{\tilde{p}^{k+1}_i-\tilde{p}^k_i}{\delta \tau}+b\frac{\tilde{p}^{k}_i-\tilde{p}^k_{i-1}}{h}=0,\qquad i=1,...,n-1.
\end{equation}With our choice $\delta\tau=\frac{1}{b}h$, the above scheme reduces to a translation
\begin{equation}\label{a1z-blow-up-tildep}
    \tilde{p}^{k+1}_i=\tilde{p}^{k}_{i-1},\qquad i=1,...,n-1.
\end{equation}For the boundary condition, we set the incoming flux to zero
\begin{equation}\label{bc-tilde-i}
    \tilde{p}^{k}_{0}=0.
\end{equation}

To solve the ODE for $M(\tau)$ \eqref{dyn-M-a1=0}, we use
\begin{equation}\label{n-ode-m-a1zero}
    M^{k+1}=M^k+\delta\tau(b\tilde{p}^k_{n-1}-1).
\end{equation} At each step we also update explicitly the time variable $\tau$
\begin{equation}\label{n-tau-a1z}
    \tau^{k+1}=\tau^k+\delta\tau,\qquad \tau^k=k\delta\tau.
\end{equation} 
This coupled schemes  $\tilde{p}$ and $M$ \eqref{a1z-blow-up-tildep}-\eqref{n-ode-m-a1zero} ensures that the total mass of the system is strictly conserved during the numerical evolution of the time dilation process, as established below.

\begin{proposition}[Discrete Mass Conservation]\label{prop:conserved-mass-a1z}
    The numerical scheme \eqref{ic-n-a1z}--\eqref{n-tau-a1z} with $\delta\tau=\frac{1}{b}h$ preserves the total mass in the following sense:
    \begin{equation}
        \sum_{i=1}^{n-1}\tilde{p}_i^{k+1}h + M^{k+1} + \tau^{k+1} = \sum_{i=1}^{n-1}\tilde{p}_i^{k}h + M^{k} + \tau^{k} = \sum_{i=1}^{n-1}\tilde{p}_i^{0}h.
    \end{equation}
\end{proposition}

\begin{proof}
    By \eqref{a1z-blow-up-tildep} and \eqref{bc-tilde-i}, we compute
    \begin{equation}
        \sum_{i=1}^{n-1}\tilde{p}_i^{k+1}h=\sum_{i=1}^{n-1}\tilde{p}_{i-1}^{k}h=\sum_{i=1}^{n-2}\tilde{p}_{i}^{k}h=\left(\sum_{i=1}^{n-1}\tilde{p}_{i}^{k}h\right)-\tilde{p}^{k+1}_{n-1}h.
    \end{equation} And \eqref{n-ode-m-a1zero} implies, when $\delta\tau=\frac{1}{b}h$,
    \begin{equation}
        M^{k+1}-M^k=h\tilde{p}^k_{n-1}-\delta\tau.
    \end{equation}Combining these with $\tau^{k+1}-\tau^k=\delta\tau$, we obtain the desired conservation. We also recall initially $M^0$ and $\tau^0$ are set to zero \eqref{initialize-tau}.
\end{proof}

\paragraph{Stop Criteria} We stop the numerical time-dilation dynamics if 
\begin{equation}\label{stop-blowup-criteria}
M^k+\delta\tau(b\tilde{p}^k_{n-1}-1)<0.
\end{equation} This is the discrete counterpart of $M(\tau)<0$ \eqref{def-Deltatau-a1=0}. Note that if we continue to evolve the numerical time-dilation dynamics, then the left hand side of \eqref{stop-blowup-criteria} would be $M^{m+1}$ in the next step.

When \eqref{stop-blowup-criteria} is met, we exit the dynamics in $\tau$ and restart the classical dynamics with a numerical post-blowup profile, which consists of two parts: a continuous density part represented by the vector $\tilde{\mathbf{{p}}}
^k_i$, and a Dirac mass part at $V_R$ with mass
\begin{equation}   \label{def-tauh-a1z} 
\Delta\tau_h:=M^k+\tau^k.
\end{equation} We choose the above mass (instead of taking $\Delta\tau_h$ to be $\tau^m$) to obtain exact mass conservation (c.f. Prop.~\ref{prop:conserved-mass-a1z}). The following proposition shows that the `continuous' part will not satisfy the numerical blow-up criteria \eqref{cr-blow-up-a0=0} and thus rules out the possibility of an immediate blow-up afterwards.
\begin{proposition}\label{prop:post-blowup} 
 For $(\tilde{\mathbf{{p}}}^k,M^k,\tau^k)$ evolving via \eqref{ic-n-a1z}-\eqref{a1z-blow-up-tildep}-\eqref{n-ode-m-a1zero}-\eqref{n-tau-a1z}, when \eqref{stop-blowup-criteria} is satisfied for the first time at step $k$ , we have
\begin{equation}\label{post-blowup-well}
     b\tilde{p}^{k}_{n-1}<1.
 \end{equation}
\end{proposition}
\begin{proof}
The fact that \eqref{stop-blowup-criteria} is satisfied for the first time for at step $k$ implies that $M^{k'}\geq0$ for all $k'\leq k$. In particular $M^{k}\geq0$. Thus \eqref{stop-blowup-criteria} implies \eqref{post-blowup-well}.
\end{proof}

\subsection{After Blow-up: Dealing with the Post-blowup Dirac Mass}\label{subsec:postbl}

At the continuous level, the continuation of dynamics after a blow-up is theoretically straightforward: the system restarts with a mixed initial profile containing a smooth density and a Dirac mass at $V_R$, as derived in \eqref{a1z-pobl-p}. However, resolving this singular initial data on a fixed Eulerian mesh presents a numerical challenge. The Dirac mass represents a sub-grid scale structure; simply projecting it onto the grid points via $$p_{i_R} \leftarrow p_{i_R} + \Delta\tau_h/h$$ introduces severe discretization errors and can destabilize the subsequent time-stepping.

To address this multiscale incompatibility, we adopt a {hybrid semi-analytical decomposition}. The core idea is to separate the singular component from the regular distribution and evolve it analytically during its initial diffusion phase. Specifically, we rely on the observation that, in the short time following the reset, the evolution of the Dirac mass is dominated by local diffusion and drift, closely resembling the Green's function of the Fokker-Planck operator on the free space.

Accordingly, we decompose the total probability density into two components:
\begin{enumerate}
    \item A \textbf{grid-based density} $(p_i^m)$ that tracks the regular part of the population.
    \item A \textbf{Lagrangian Gaussian packet} parameterized by its mass $\mathfrak{m}$, mean $\beta(t)$, and variance $\gamma(t)$, which approximates the spreading of the reset population.
\end{enumerate}
This Gaussian packet is evolved via a set of Ordinary Differential Equations (ODEs) derived from the exact moments of the Fokker-Planck equation. It remains distinct from the grid until it has diffused sufficiently to be resolved by the mesh size $h$ (i.e., when $\sqrt{\gamma} \sim h$), at which point it is seamlessly re-projected onto the Eulerian grid.

\paragraph{Derivation}
Suppose the solution blows up at time $t_*$. The the restart dynamics after the blow-up can be stated as follows
\begin{align}\label{eq:fp-classical-rs}
            \p_{t}p+\p_v([-v+bN(t)]p)&=a\p_{vv}p+N(t)\delta_{v=V_R},\qquad &t>t_*,v<V_F,\\\label{zerobc-classical-rs}
p(t,V_F)&=0,\qquad N(t)=-a\p_vp(t,V_F)\geq 0,\qquad &t>t_*,\\ p((t_*)^+,v)&=\tilde{p}(\Delta\tau,v)+\Delta\tau \delta_{v=V_R}\qquad &v\leq V_F,\label{ic:fp-classical-rs}
\end{align}We consider the decomposition $p(t,v)=p_1(t,v)+p_2(t,v)$, where $p_1$ and $p_2$ are govern by (c.f. \cite[Section 3.2]{dou2024noisy} and \cite{JGLiuZZ22} for decompositions of similar spirits)
\begin{equation}
    \begin{dcases}
        \p_{t}p_1+\p_v([-v+bN(t)]p_1)=a\p_{vv}p_1+N_1(t)\delta_{v=V_R},\qquad t>t_*,v<V_F,\\
        p_1(t,V_F)=0,\qquad N_1(t):=-a\p_vp_1(t,V_F)\geq 0,\qquad t>t_*,\\ p_1((t_*)^+,v)=\tilde{p}(\Delta\tau,v),\qquad v\leq V_F,
    \end{dcases}
\end{equation} and 
\begin{equation}
    \begin{dcases}
        \p_{t}p_2+\p_v([-v+bN(t)]p_2)=a\p_{vv}p_2+N_2(t)\delta_{v=V_R},\qquad t>t_*,v<V_F,\\
        p_2(t,V_F)=0,\qquad N_2(t):=-a\p_vp_1(t,V_F)\geq 0,\qquad t>t_*,\\ p_2((t_*)^+,v)=\Delta\tau \delta_{v=V_R},\qquad v\leq V_F,
    \end{dcases}
\end{equation} The two equations are almost the same except for the initial data. They are coupled by the total firing rate $N(t)$, which can also be decomposed as
\begin{equation}
    N(t)=N_1(t)+N_2(t).
\end{equation}

Now we apply approximations. Intuitively, the profile of $p_2$ shall be localized in space/voltage and away from $V_F$ in the short time, since it starts with a Dirac mass at $V_R$. This inspires us to consider two approximations in the short time. First, we assume the contribution of $p_2$ to the total firing rate is negligible
\begin{equation}
    N_2(t)\approx 0,\qquad N(t)\approx N_1(t).
\end{equation} This makes $p_1$ satisfy a self-contained equation (no longer dependent on $p_2$, can be solved independently).
\begin{equation}
    \begin{dcases}
        \p_{t}p_1+\p_v([-v+bN_1(t)]p_1)=a\p_{vv}p_1+N_1(t)\delta_{v=V_R},\qquad t>t_*,v<V_F,\\
        p_1(t,V_F)=0,\qquad N_1(t):=-a\p_vp_1(t,V_F)\geq 0,\qquad t>t_*,\\ p_1((t_*)^+,v)=\tilde{p}(\Delta\tau,v),\qquad v\leq V_F,
    \end{dcases}
\end{equation} Secondly, for the same reason, we assume in the short time the evolution of $p_2$ can be well-approximated by the Fokker-Planck equation on the whole space
\begin{equation}
    \begin{dcases}
        \p_{t}p_2+\p_v([-v+bN(t)]p_2)=a\p_{vv}p_2,\qquad t>t_*,v\in\mathbb{R},\\ p_2((t_*)^+,v)=\Delta\tau \delta_{v=V_R},\qquad v\leq V_F,
    \end{dcases}
\end{equation} which can be (semi-)explicitly solved as (e.g. \cite{risken1996fokker}) the following evolving Gaussian function 
\begin{equation}\label{formula-p2}
    p_2(t,v)=\frac{\Delta\tau}{\sqrt{2\pi \gamma(t)}}\exp\left(-\frac{(v-\beta(t))^2}{2\gamma(t)}\right),\qquad t>t^*, v\in\mathbb{R},
\end{equation} where the mean $\beta(t)$ and the variance $\gamma(t)$ satisfy the following ODE system
\begin{equation}\label{ODE-Gaussian}
    \begin{dcases}
        \frac{d}{dt}\beta(t)=bN(t)-\beta(t),\qquad t>t_*,\\
        \frac{d}{dt}\gamma(t)=2a-2\gamma(t),\qquad t>t_*,\\
    \beta(t_*)=V_R,\qquad\qquad \gamma(t_*)=0.
    \end{dcases}
\end{equation}
Indeed, the core assumption of both approximations is $N_2(t)\approx 0$, which is intuitively valid in the short time.

\paragraph{Schemes}
Now we detail the numerical schemes based on the above decomposition. Instead of representing the solution by just a vector $(p^m_i)_{i=1}^{n-1}$, we extend the representation to $({\mathbf{p}}^m,\mathfrak{m},\beta^m,\gamma^m)$, where the latter three variables represent the evolving Gaussian \eqref{formula-p2}: $\mathfrak{m}$ for its mass, $\beta^m$ for the mean and $\gamma^m$ for the variance. Such a representation is initialized as
\begin{equation}\label{init-putback}
p^0_i:=\tilde{p}_i^k,\quad \mathfrak{m}:=\Delta\tau_h,\quad \beta^0=V_R,\quad \gamma^0=0,
\end{equation} where $\tilde{p}_i^k$ and $\Delta\tau_h$ are taken from the last step of the time dilation dynamics \eqref{stop-blowup-criteria}-\eqref{def-tauh-a1z}. In particular the stop criteria for the blow-up \eqref{stop-blowup-criteria} is met.

At each step, the vector $(p^m_i)_{i=1}^{n-1}$ is evolved in the exactly same way as described in Section \ref{subsubsec:cl-a1z}, according to which we can compute the firing rate $N_h^m$. The mass of the Gaussian $\mathfrak{m}$ will remain constant during the evolution. For $\beta^m$ and $\gamma^m$, we numerically solve the ODE \eqref{ODE-Gaussian} with the exponential-time-integral method
\begin{equation}\label{ODE-bg-scheme}
    \begin{dcases}
        \beta^{m+1}=e^{-\delta t}\beta^m+(1-e^{-\delta t})bN_h^m,\\
        \gamma^{m+1}=e^{-2\delta t}\gamma^m+(1-e^{-2\delta t})a.
    \end{dcases}
\end{equation}

\paragraph{Criteria for Re-projection}
The semi-analytical phase is transient. The Lagrangian packet is assimilated back into the Eulerian density vector $\mathbf{p}^m$ as soon as one of the following two conditions is met, ensuring both numerical accuracy and physical validity.

\begin{enumerate}
    \item \textbf{Resolution Sufficiency:} The primary criterion for re-projection is geometric. Once the variance of the Gaussian packet becomes comparable to the grid resolution, the packet is sufficiently smooth to be resolved by the fixed mesh without introducing significant aliasing errors. We define this threshold as:
    \begin{equation}\label{cr-2}
        \gamma^m \geq C_{\text{res}} h^2,
    \end{equation}
    where $C_{\text{res}}$ is a constant of order $O(1)$. Since the variance grows linearly as $\gamma(t) \sim 2at$, and the time step typically scales as $\delta t \sim O(h^2)$, this condition is usually satisfied within a few time steps.

    \item \textbf{Boundary Interaction:} The analytical approximation \eqref{formula-p2} assumes evolution in free space, ignoring the absorbing boundary at $V_F$. If the tail of the Gaussian packet approaches the threshold significantly, this assumption breaks down. To prevent physical inconsistency, we enforce re-projection if the flux contribution becomes non-negligible:
    \begin{equation}\label{cr-1}
         \frac{\mathfrak{m}}{\sqrt{2\pi \gamma^m}}\exp\left(-\frac{(V_F-\beta^m)^2}{2\gamma^m}\right) \geq \epsilon,
    \end{equation}
    where $\epsilon$ is a small tolerance parameter.
\end{enumerate}

Upon triggering either criterion, the Gaussian mass $\mathfrak{m}$ is mapped onto the grid points $v_i$. To enforce exact mass conservation at the discrete level, we introduce a normalization factor $\mathcal{Z}$:
\begin{equation}\label{put-back-1}
    p^m_i \leftarrow p^m_i + \frac{\mathfrak{m}}{\mathcal{Z}} \frac{1}{\sqrt{2\pi \gamma^m}}\exp\left(-\frac{(v_i-\beta^m)^2}{2\gamma^m}\right),
\end{equation}
where $\mathcal{Z} := \sum_{j=1}^{n-1} \frac{1}{\sqrt{2\pi \gamma^m}}\exp\left(-\frac{(v_j-\beta^m)^2}{2\gamma^m}\right) h$. This step completes the cycle, returning the system fully to the macroscopic representation described in Section \ref{subsubsec:cl-a1z}.

    Theoretically, a subsequent blow-up event could occur before these re-projection criteria are met. In such a scenario, the algorithm can be naturally extended to track multiple Gaussian packets simultaneously. However, in all numerical regimes explored in this work, the rapid regularization of the Dirac mass ensures that re-projection occurs well before the onset of the next synchronization event.

\subsection{Summary of the Algorithmic Protocol}

We conclude this section by synthesizing the proposed multiscale framework into a cohesive algorithmic protocol. The solver operates on a spatially uniform mesh over the truncated domain $[V_{\min}, V_F]$ and alternates between three distinct phases.

\begin{enumerate}
    \item \textbf{Phase I: Macroscopic Integration ($t$-scale).} 
    In the sub-critical regime, the density vector $\mathbf{p}^m$ is evolved using the semi-implicit structure-preserving scheme \eqref{linear-choice-1}. 
    \begin{itemize}
        \item \textit{Monitoring:} At each step, we check the mesh-independent stability criterion \eqref{numerical-dbcl}: $b p_{n-1}^m < 1$.
        \item \textit{Transition:} If the criterion is violated, the macroscopic evolution is suspended, and the solver transitions to Phase II, initializing the mesoscopic state.
    \end{itemize}

    \item \textbf{Phase II: Mesoscopic Resolution ($\tau$-scale).} 
    The dynamics are lifted to the dilated timescale to resolve the singularity.
    \begin{itemize}
        \item \textit{Evolution:} The density $\tilde{\mathbf{p}}^k$ is updated via the exact upwind shift \eqref{ic-n-a1z}, while the auxiliary mass $M^k$ and time $\tau^k$ are integrated via \eqref{n-ode-m-a1zero}--\eqref{n-tau-a1z}.
        \item \textit{Transition:} The phase concludes when the look-ahead termination criterion \eqref{stop-blowup-criteria} is met. The blow-up size $\Delta \tau_h$ is computed, and the solver transitions to Phase III.
    \end{itemize}

    \item \textbf{Phase III: Hybrid Assimilation (Transition).} 
    The system restarts in the $t$-scale with a hybrid representation: a grid-based density and a Lagrangian Gaussian packet.
    \begin{itemize}
        \item \textit{Evolution:} The density vector evolves as in Phase I, while the Gaussian packet parameters $(\beta, \gamma)$ evolve via the ODEs \eqref{ODE-bg-scheme}.
        \item \textit{Assimilation:} Once the packet satisfies the resolution condition $\gamma \geq C_{\text{res}} h^2$ or interacts with the boundary, it is re-projected onto the Eulerian grid via the conservative mapping \eqref{put-back-1}. The algorithm then fully reverts to Phase I.
    \end{itemize}
\end{enumerate}

\begin{remark}[Local vs. Global Dilated Time]
    It is important to clarify that in the algorithmic implementation, the $\tau$ timescale is employed \textit{locally}: specifically, we initialize the time-dilation dynamics at $\tau=0$ whenever the solver enters Phase II. This local resetting isolates each blow-up event numerically, simplifying the transport logic. 
    
    However, to visualize the long-term dynamics continuously (as will be shown in Section \ref{sec:experiments}), we construct a cumulative \textit{global} dilated time $\tau_{\text{global}}$. The global clock advances by $N_h \delta t$ during the macroscopic phases and by $\delta\tau$ during the mesoscopic phase, providing a unified timeline that resolves the entire history of multiple firing events.
\end{remark}

\section{Numerical Experiments and Explorations}\label{sec:experiments}



In this section, we present a series of numerical experiments to validate the multiscale time-dilation scheme and to explore the complex dynamics of the mean-field FP equations \eqref{eq:intro-fp}. Our investigation is structured to address both the algorithmic performance and the underlying physical modeling.

In Section \ref{pdedemo}, we first conduct a systematic test of the PDE solver. By varying the connectivity parameter $b$, we demonstrate the solver's capability to capture different dynamical regimes, ranging from steady-state distributions to periodic synchronized bursts (MFEs). These tests serve as a verification of the numerical consistency and robustness of the time-dilation framework in handling the inherent singularities of the model dynamics.

In Section \ref{sec:comparison}, we then focus comparing the mean-field PDE system and the finite-size particle system. Our goal is to move beyond simple validation and delve into two fundamental questions:
\begin{enumerate}
    \item \textbf{Validity Range:} Under what conditions (e.g., population size $N_p$ and connectivity $b$) does the mean-field PDE serve as a reliable approximation for the discrete particle dynamics, and where do the finite-size effects lead to qualitative departures?
    \item \textbf{Quantitative Advantages:} In regimes where the mean-field approximation is valid, what are the specific advantages of the PDE solver over the particle solver in terms of eliminating statistical noise and enhancing computational efficiency, especially near critical blow-up points?
\end{enumerate}

Through these explorations, we aim to provide a comprehensive picture of how the multiscale PDE approach facilitates a deeper understanding of the neuronal population dynamics that is otherwise computationally prohibitive for traditional particle-based methods.

\subsection{Numerical Tests on PDE Solver}\label{pdedemo}

In this subsection, we validate the effectiveness of the proposed multiscale numerical framework based on time-dilation in handling the mean-field Integrate-and-Fire model through a series of representative numerical experiments. We investigate the dynamical evolution of the system under different excitatory coupling strengths $b$, aiming to demonstrate that the algorithm not only accurately captures the classic steady-state distribution but also achieves the proper physical extension of solutions through automatic time-scale switching when the firing rate $N(t)$ undergoes a singularity (blow-up). These experiments cover various dynamical regimes ranging from steady-state convergence without blow-ups to single blow-ups and to complex periodic multiple firing events, thereby systematically evaluating the algorithm's robustness in handling discontinuous transitions and mass conservation.

Our numerical setting is consistent with previous theoretical results. For $b>0$ large enough, every solution blows up \cite{caceres2011analysis,roux2021towards}. While for $b$ small but positive, both the finite time blow-up and the global existence are possible depending on the initial data \cite{caceres2011analysis,carrillo2013CPDEclassical,DelarueIRT2015AAP,roux2021towards}.

\subsubsection{Sub-critical Regime: Convergence to Steady State}
We begin by presenting the numerical results for $0\leq b < 1$. For this case, an appropriate initial condition was chosen so that the PDE system does not blow up. The membrane potential distribution converges to a steady state after a transient. Here, the cases $b = 0$ and $b = 1/2$ are presented. For both cases, the total simulation time was set to $10$. 

Numerical simulations show that regardless of the initial distribution, the system eventually converges to a continuous steady-state density distribution $p(v)$. The corresponding firing rate $N(t)$ gradually tends toward a constant after initial fluctuations, indicating that the system enters a steady-state, low population firing mode. The auxiliary function $M(\tau)$ remains zero throughout the evolution process (see the bottom right panels of Fig.~\ref{figpdeb0} and Fig.~\ref{figpde1/2}). This confirms that the system does not blow up  in this scenario. 
\begin{figure}[H]
    \centering
    \includegraphics[width=0.7\linewidth]{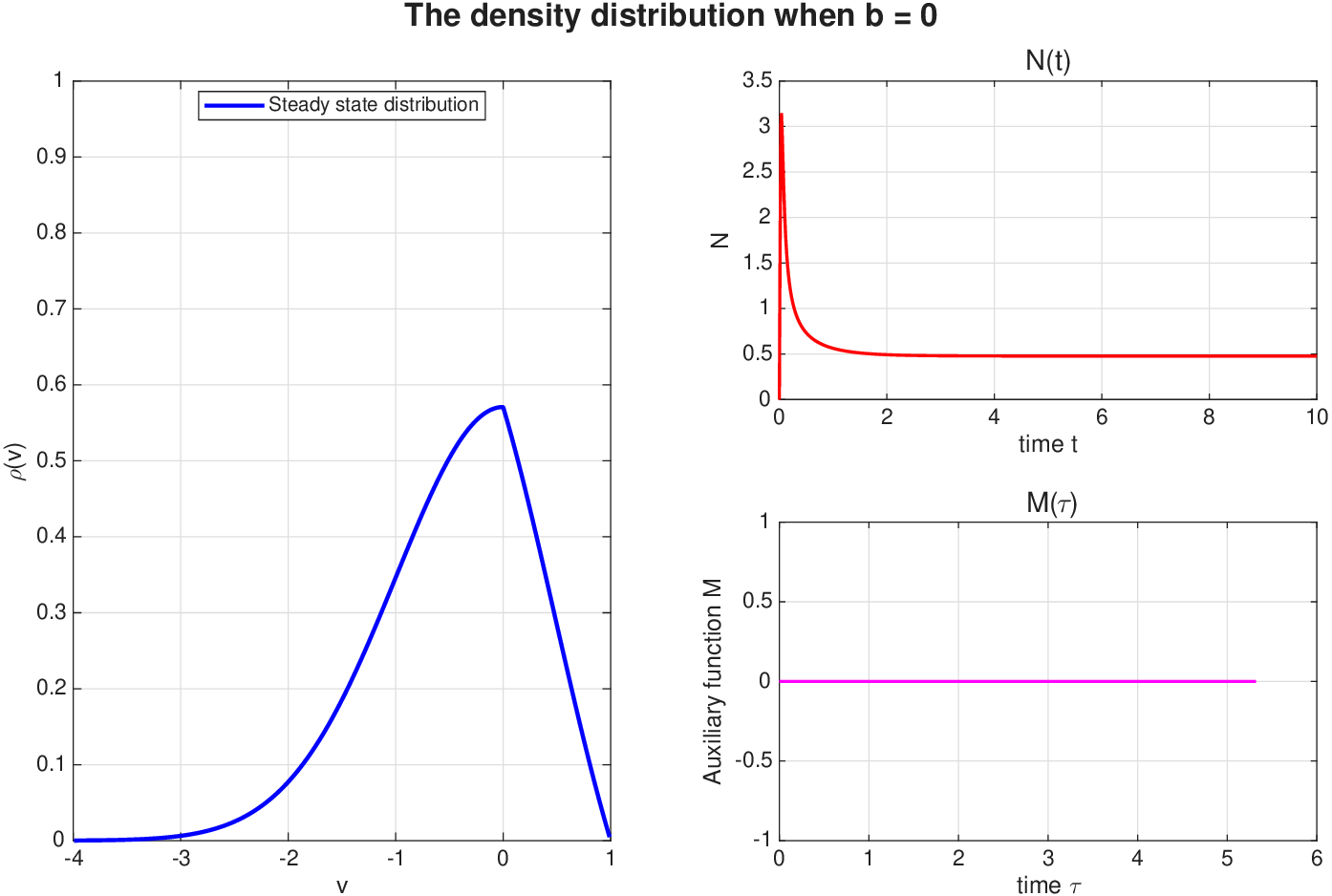}
    \caption{The case $b = 0$. The left column displays the steady-state density profiles; the top-right panels show the corresponding $N(t)$, and the bottom-right panels present the auxiliary function $M(\tau)$ on the $\tau$-time scale.}
    \label{figpdeb0}
\end{figure}
\begin{figure}[H]
    \centering
    \includegraphics[width=0.7\linewidth]{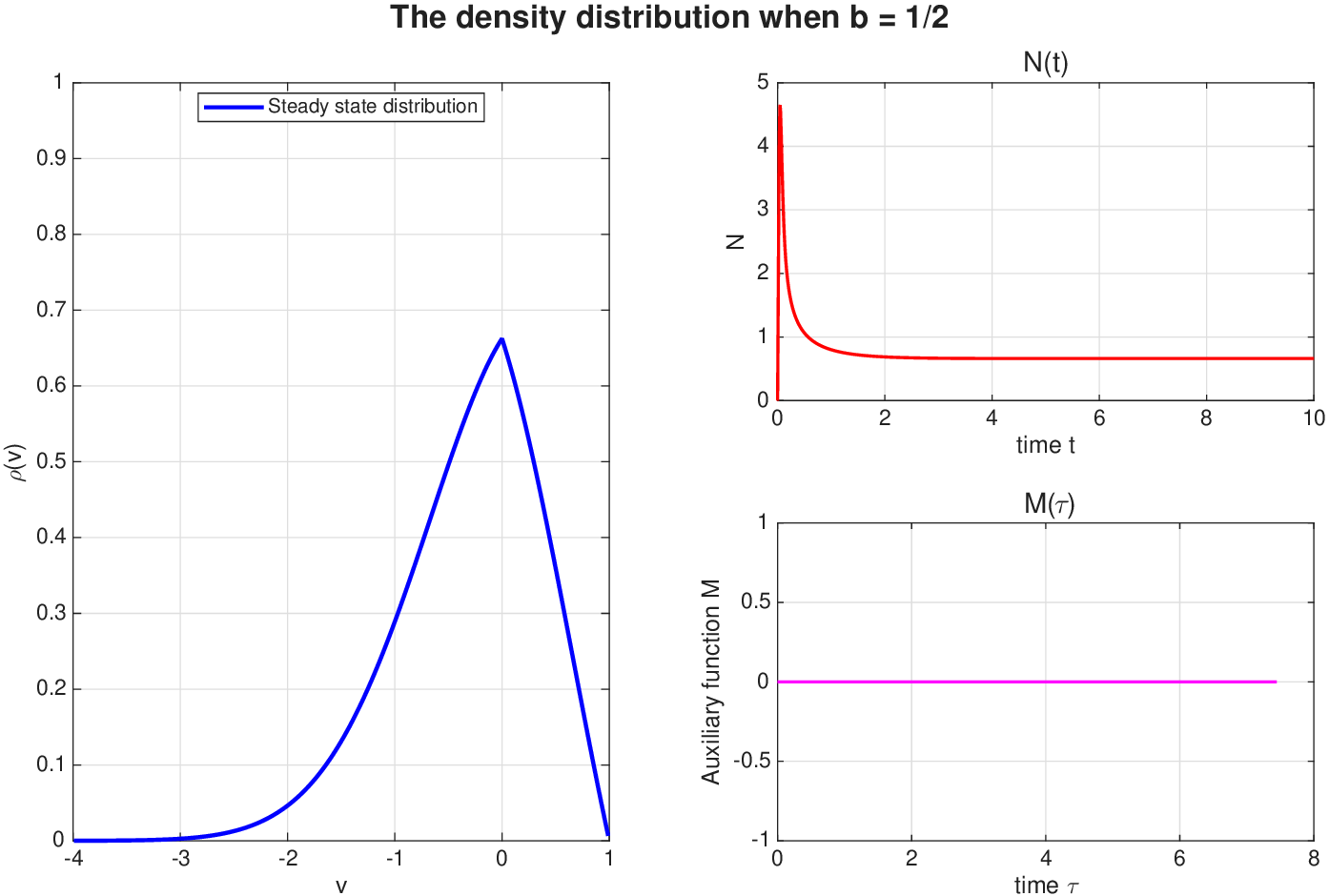}
    \caption{The case $b = 1/2$. The left column displays the steady-state density profiles; the top-right panels show the corresponding $N(t)$, and the bottom-right panels present the auxiliary function $M(\tau)$ on the $\tau$-time scale.}
    \label{figpde1/2}
 \end{figure}

\subsubsection{Critical Regime: Single MFE and Re-stabilization}
When the excitation strength increases to $b=1$, the system exhibits dynamical characteristics significantly different from the low-excitation case. Starting from the same initial condition for $b < 1$, this same datum induces a single blow-up when $b = 1$. When the PDE blows up, $N(t)$ becomes infinite; for better visualization, its value is set to $1000$ at those instants.

\begin{figure}[H]
    \centering
    \includegraphics[width=0.8\linewidth]{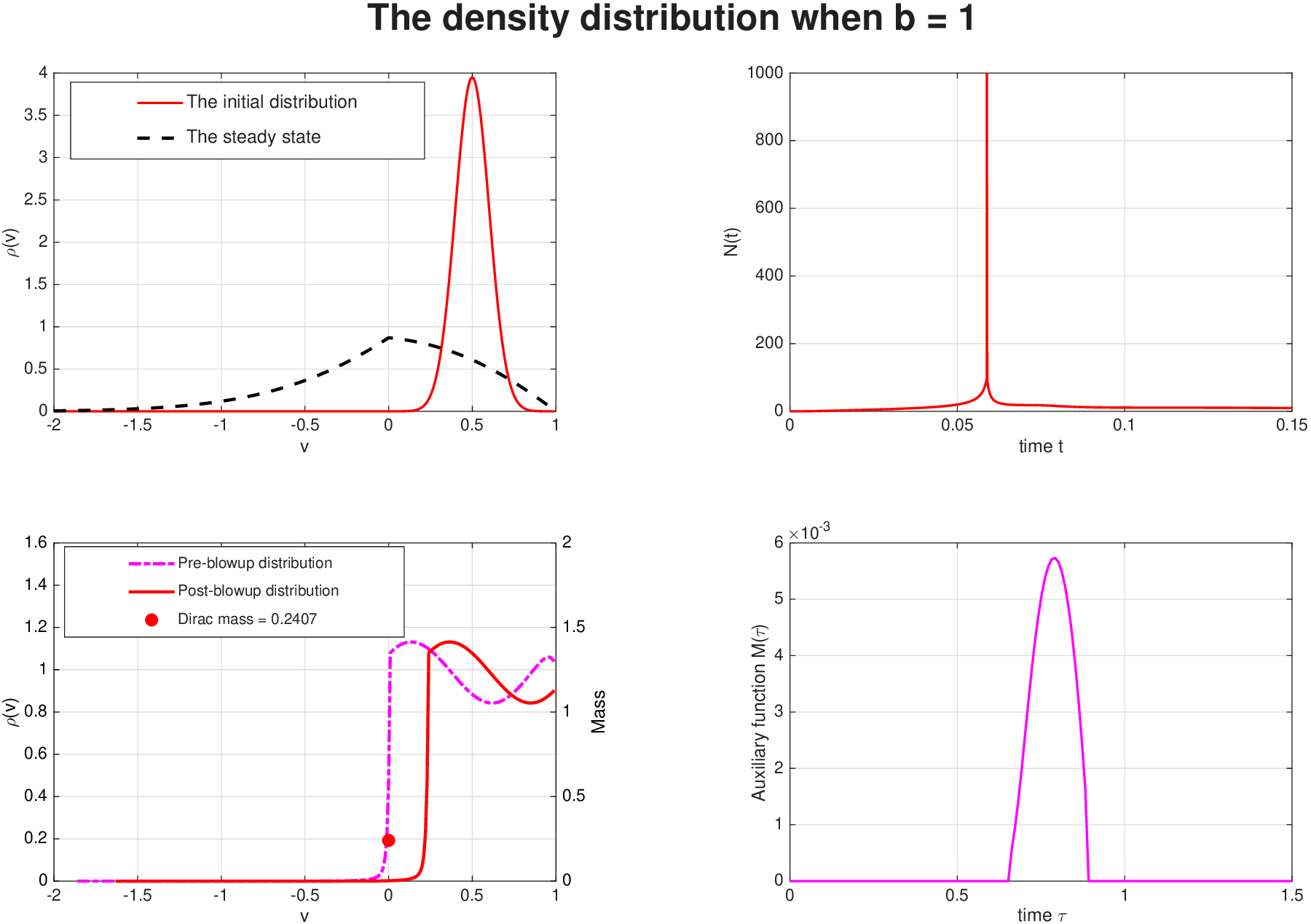}
    \caption{Evolution of the probability density $p(v, t)$ in the critical regime ($b = 1$). This figure illustrates the process where the system concentrates from a smooth initial distribution toward the threshold $V_F = 1$ at the critical connectivity strength. Driven by excitatory feedback, the system triggers a single synchronization. The sharp changes in density during the explosion are precisely captured via the time-dilation technique. Following the synchronization, the system returns to a stable regime, with the probability density gradually approaching a new steady-state distribution. Top-left: The initial distribution and the steady state distribution. Down-left: The pre-blowup and post-blowup distribution and the corresponding Dirac mass. Top-right: The profile of $N(t)$. Down-right: The profile of auxiliary function $M(\tau)$.} 
    \label{figpdeb1}
\end{figure}
At a specific instant early in the evolution, the rapid accumulation of potential near the firing threshold $V_F$ causes $N(t)$ to tend toward infinity. At this point, the algorithm successfully identifies the blow-up and automatically switches to the dilated time scale $\tau$. Under the dilated scale, the originally instantaneous synchronized firing activity blow-up is expanded into a continuous evolution. The algorithm accurately calculates the blow-up mass $\Delta m$ and forms a Dirac mass at the reset potential $V_R$. After the blow-up ends, the system returns to the natural time scale $t$. The Dirac mass gradually smooths out under the effect of diffusion, and the system eventually stabilizes at a new steady state after undergoing a significant reorganization, verifying the algorithm's ability to handle a single blow-up.

\subsubsection{Super-critical Regime: Periodic Synchronization Patterns}
With strong excitatory coupling (e.g., $b=2$ and $b=10$), the system enters a sustained synchronous periodic oscillation state. We continue to use the same initial condition for $b>1$, and periodic blow-up solutions emerge; here, the cases $b = 2$ and $b = 10$ areare presented. 

Fig.~\ref{fig:pdeb2} is for the case $b=2$ and shows that the density profiles immediately before the second and third blow-ups are virtually identical, and the auxiliary function $M(\tau)$ becomes periodic starting from the second blow-up. The panel also displays the post-blowup Dirac mass together with the particle density; their masses sum to $1$, which agrees with the theoretical guarantee of mass conservation of our algorithm.

Same can be said to Fig.~\ref{fig:pdeb10} which is for the case $b=10$. A notable difference is in the post-blowup profile: the magnitude of the Dirac mass part is now equal to $1$ while the value of the density function is $0$. This indicates a full synchronization where the whole population is involved, which is stronger than the case $b=2$. 

Results show that the firing rate $N(t)$ presents regular blow-up pulses. Each singularity of $N(t)$ corresponds to the synchronous firing and resetting of a large portion of the neuron population. Faced with such high-frequency singularity transitions, the numerical framework demonstrates high stability. The algorithm accurately captures the occurrence time and magnitude of each MFE and handles the subsequent formation and diffusion of the Dirac mass. The periodic ``bottoming out" and recovery of the $M(\tau)$ function reveals the cyclic mechanism between internal charge accumulation and synchronous reset within the system (see charts related to $b=10$). 
\begin{figure}[H]
    \centering
    \includegraphics[width=0.8\linewidth]{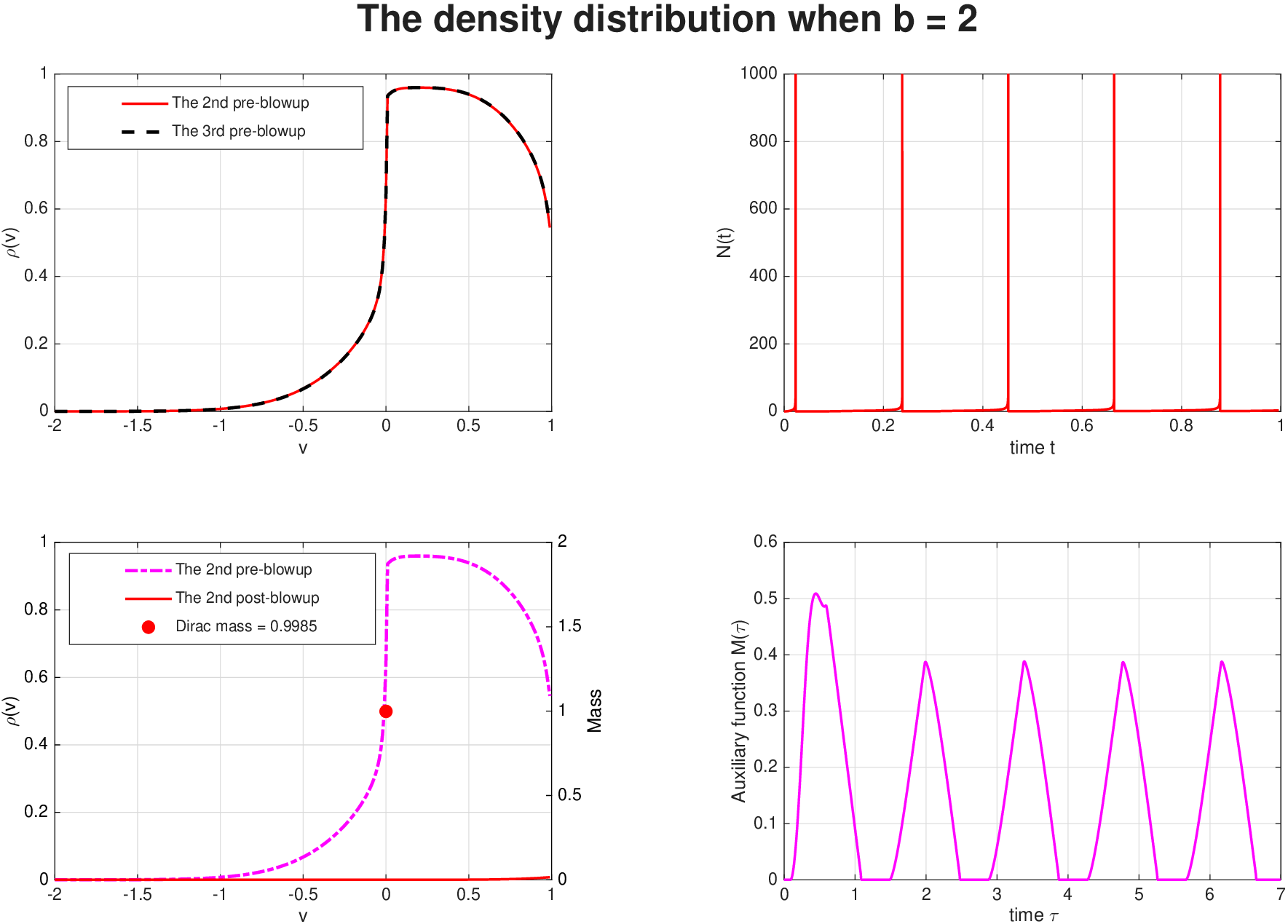}
    \caption{Dynamical behavior of the probability density $p(v,t)$ at moderate connectivity strength ($b=2$). As the connectivity strength increases to $b=2$, the system exhibits more pronounced nonlinear characteristics. The figure demonstrates the effectiveness of the numerical algorithm in handling the formation and evolution of the Dirac mass at the moment of blow-up, reflecting the rapid recovery process of the system after a massive discharge event. Top-left: The second and the third pre-blowup profiles. Down-left: The second pre-blowup and post-blowup distribution and the corresponding Dirac mass. Top-right: The profile of $N(t)$. Down-right: The profile of auxiliary function $M(\tau)$.}
    \label{fig:pdeb2}
\end{figure}

\begin{figure}[H]
    \centering
    \includegraphics[width=0.8\linewidth]{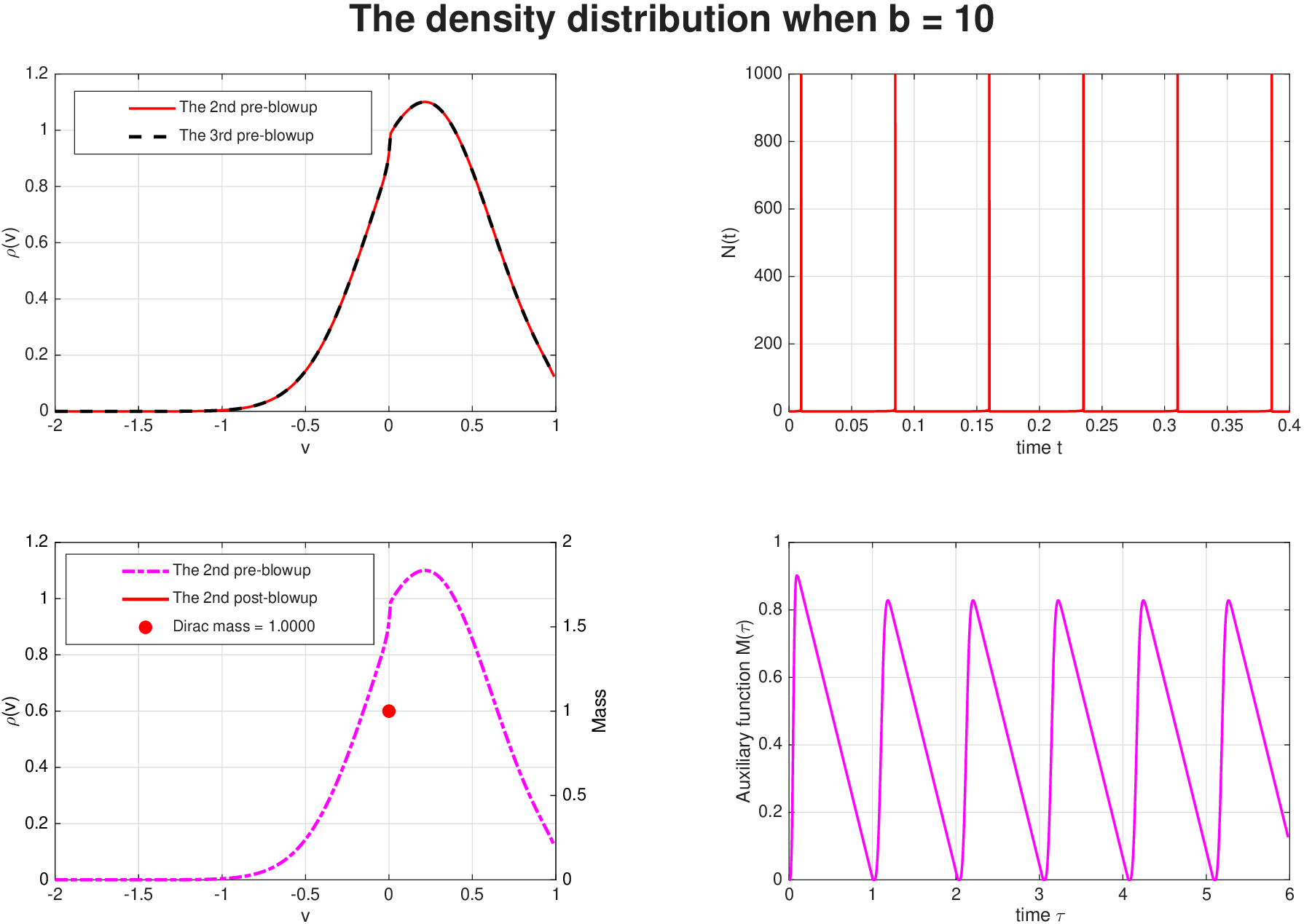}
    \caption{Phenomenon of multiple synchronizations in the strong coupling regime ($b = 10$). At high connectivity strength, the system enters a strongly synchronized firing mode. This figure displays the cyclic dynamics of the probability density $p(v, t)$ through successive explosions and resets. Due to the large value of $b$, the feedback from a single discharge is sufficient to trigger subsequent blow-up events within a short time frame. The evolution path clearly shows the ``blowup-recovery-blowup" cycles, validating the capability of our multiscale framework to handle multiple singularities. Top-left: The second and the third pre-blowup profiles. Down-left: The second pre-blowup and post-blowup distribution and the corresponding Dirac mass. Top-right: The profile of $N(t)$. Down-right: The profile of auxiliary function $M(\tau)$.}
    \label{fig:pdeb10}
\end{figure}


\subsection{Comparison between PDE Solver and Particle Solver}\label{sec:comparison}

In this section, we provide a comprehensive quantitative assessment of the proposed mean-field PDE solver against the microscopic particle simulation. Our objective is two-fold: first, to {investigate} the convergence of the particle system to the mean-field limit as $N_p \to \infty$; and second, to demonstrate the superior computational efficiency and robustness of the PDE solver as an approximation tool, particularly in regimes where the microscopic dynamics becomes numerically stiff.

{Theoretical results on the limit $N_p \to \infty$ is only qualitative \cite{delarue2015particle}, and there is no rigorous estimates on the quantitative discrepancy between the finite-size particle system and the mean-field limit.} We expect it to scales as $O(C N_p^{-\alpha})$ and conjecture that the prefactor $C$ depends critically on the coupling strength $b$ and the network activity $N(t)$. This dependence motivates our experimental design across different regimes:
\begin{itemize}
    \item \textbf{Weak Coupling ($b=1/2$):} The system remains regular. We expect $C$ to be small, allowing the particle system to converge rapidly with standard time steps.
    \item \textbf{Strong Coupling ($b=2, 10$):} The system exhibits periodic synchronization. In this regime, the microscopic dynamics become extremely stiff near the blow-up moments, requiring the particle solver to employ time steps significantly smaller than $1/N_p$ to resolve the fast cascade. In contrast, our PDE solver, empowered by the time-dilation transformation, resolves these singularities using reasonably small time steps, highlighting its advantage as a robust approximation.
    \item \textbf{Critical Regime ($b=1$):} The system sits on the verge of a single blow-up, representing the most sensitive case for finite-size effects.
\end{itemize}

To quantify the agreement, we employ the accumulated $L^1$-error of the first-order moment (population-averaged voltage) over the simulation interval $[0, T]$. The first moments {for the PDE and for the particle system} are defined as, respectively:
\begin{equation}\label{eq:moment_def}
    L_{\rm{pde}}(t) = \int_{V_{\min}}^{V_F} v p(v,t) \,dv, \qquad L_{\rm{ps}}(t) = \frac{1}{N_p} \sum_{i=1}^{N_p} V_t^i.
\end{equation}
The accumulated error metric is given by:
\begin{equation}\label{eq:L_error_def}
    L_{\rm{error}} := \int_0^T |L_{\rm{pde}}(t) - L_{\rm{ps}}(t)| \, dt \approx \sum_{m} |L_{\rm{pde}}^m - L_{\rm{ps}}^m| \Delta t.
\end{equation}
We deliberately choose this time-integrated metric because it is {less sensitive to the discrepancy of the synchronization time, compared to the $L^{\infty}_t$ norm. As the synchronization is a discontinuous event in time, a small mismatch in its time can lead to an $O(1)$ error, if we use the $L^{\infty}_t$ norm in time. Using this time-integrated metric ensures that an $O(\eps)$ discrepancy in period still leads to an $O(\eps)$ error. Indeed, the failure of the $L^{\infty}_t$ norm to capture the temporal error of a discontinuous event echoes the motivation for Skorokhod to introduce the M1 topology \cite{skorokhod1956limit}, which is extensively used in the related analysis \cite{delarue2015particle}.}

Unless otherwise specified, the time step for the particle simulation is set to $\delta t = 1/N_p$, and we plot the first-order moment trajectories based on a single realization of the particle system. However, for small population sizes (e.g., $N_p=1000$) where stochastic fluctuations are pronounced, we calculate the ensemble average over multiple independent trajectories (realizations) to reduce random noise and isolate the systematic finite size effect. 

{Additionally, it worth noting that for all the tabulated particle results in Section~\ref{subsecb10} and ~\ref{subsecb2}, the particle system is simulated independently over $100$ runs with the same parameter set. For each synchronization event, we record the corresponding blow-up time in every run, and then compute the sample mean (denoted as PS mean) and sample variance (denoted as PS variance) of these blow-up times. In addition, we report the mean square error (MSE) between the particle blow-up times and the corresponding PDE blow-up times. Therefore, Tables~\ref{tab_time_b10_dt_1dN}-\ref{tab_time_b2_smalldt} provide a statistical comparison between the microscopic particle solver and the deterministic PDE solver, rather than the outcome of a single particle realization.}

\subsubsection{Weak Coupling Regime: Steady State Convergence ($b = 1/2$)}

We begin our analysis in the weak coupling regime ($b = 1/2$), where the microscopic particle model relaxes to a stable stationary state. This regime serves as a baseline to demonstrate how the PDE solver captures the collective behavior of the underlying physical system.

Figure \ref{fig:b1d2} illustrates the fundamental distinction between the two descriptions. The individual trajectories of the particle model (colored lines) inherently exhibit stochastic fluctuations around the mean-field limit. While the ensemble behavior approaches the deterministic limit as the population size $N_p$ increases, recovering a smooth profile from the microscopic model is computationally expensive, requiring the simulation of a massive number of particles to suppress the statistical noise.

In sharp contrast, the PDE solver directly computes the noise-free mean-field limit (black {solid} line). As shown in Table \ref{tab_b1d2_dt_1dN}, the discrepancy between the population average and the PDE solution scales as $O(N_p^{-1/2})$, consistent with the central limit theorem. This result highlights the significant computational advantage of the PDE approach: it bypasses the slow convergence of the Monte Carlo sampling and provides an accurate, instantaneous description of the system's collective state with minimal computational cost.

\begin{figure}[H]
    \centering
    \includegraphics[width=0.8\linewidth]{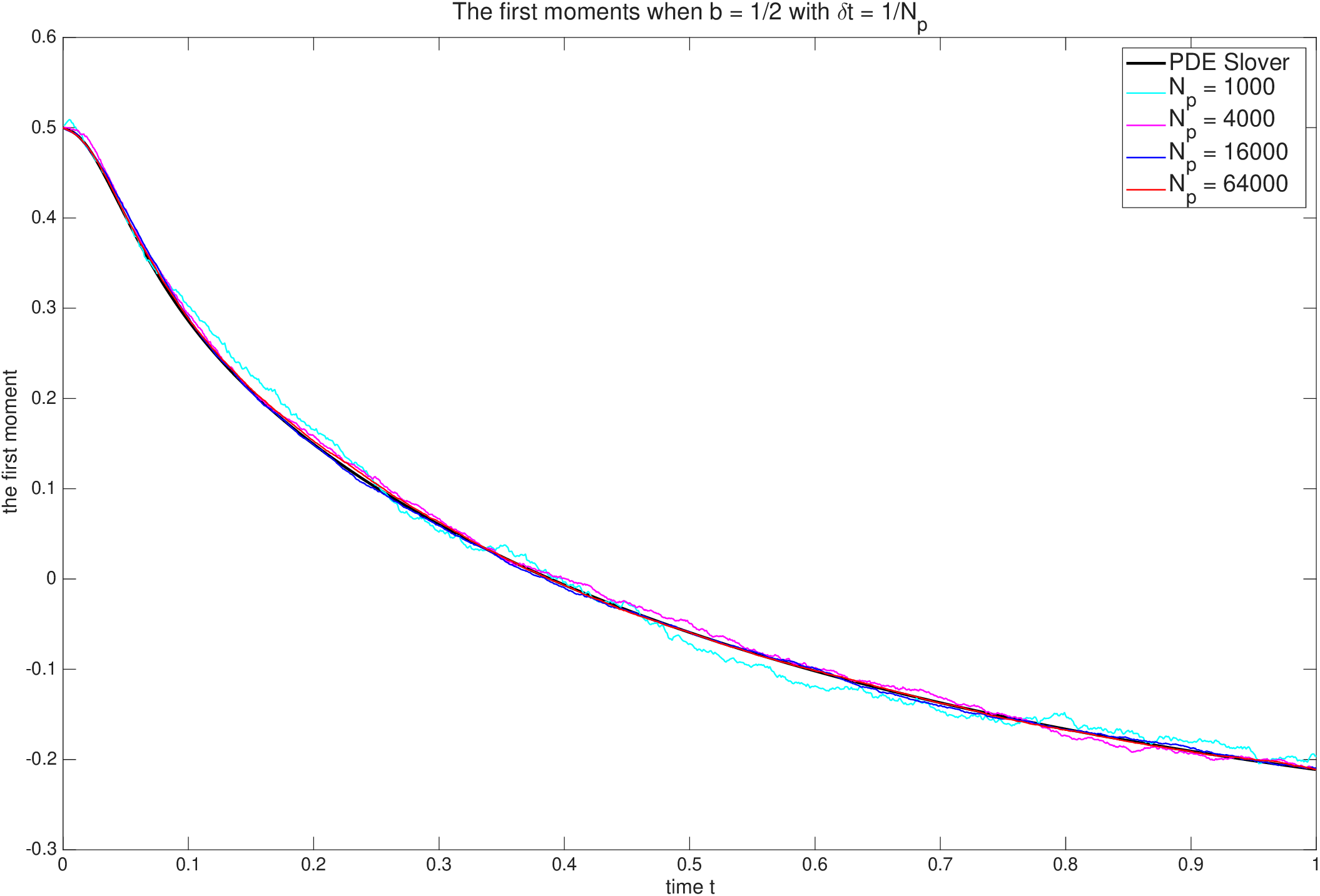}
    \caption{First-order moment trajectories for $b = 1/2$. The PDE solver (black solid line) directly captures the noise-free mean-field behavior, while the microscopic particle model (colored lines) exhibits inherent stochastic fluctuations that diminish only with increasing population size $N_p$.}
    \label{fig:b1d2}
\end{figure}

\begin{table}[htbp]
    \centering
    \caption{$L_{\rm{error}}$ of the first moment for $b=1/2$ with different population size $N_p$. The discrepancy reflects the finite-size fluctuations of the particle model, which decay at the rate $O(N_p^{-1/2})$ relative to the mean-field limit.}
    \begin{tabular}{|c|c|c|c|c|}
    \hline
    {$N_p$} & {$1,000$} & {$4,000$} & {$16,000$} & {$64,000$} \\ 
    \hline
    {$L_{\rm{error}}$} & {$9.34\times10^{-3}$} & {$4.97\times10^{-3}$} & {$2.12\times10^{-3}$} & {$1.53\times10^{-3}$} \\
    \hline
    \end{tabular}
    \label{tab_b1d2_dt_1dN}
\end{table}

\subsubsection{Strong Coupling Regime: Periodic Synchronization ($b = 10$)}\label{subsecb10}

We now turn to the strong coupling regime ($b = 10$), where the system exhibits high-frequency periodic MFEs. This regime provides a stringent test for the solvers due to the rapid timescale of the synchronization process.

\paragraph{Convergence and Numerical Stiffness}
First, we investigate the convergence behavior for large population sizes. In Table \ref{tab_time_b10_dt_1dN} and Figure \ref{fig:b10_1dN}, we observe that with the standard time step scaling $\delta t = 1/N_p$, the particle simulation exhibits a noticeable phase drift from the PDE solution after several cycles, even for $N_p = 256,000$. Specifically, the cumulative error in the blow-up times leads to a visible discrepancy in the third and fourth events. {Table~\ref{tab_time_b10_dt_1dN} quantifies this cumulative mismatch more precisely. Although the particle system still captures the qualitative periodic structure, the mean blow-up times are systematically delayed relative to the PDE prediction, and the discrepancy increases over successive events. Correspondingly, the mean squared error (MSE) also grows from one cycle to the next, indicating that the temporal error is not a purely local fluctuation but is accumulated dynamically through repeated synchronization. This behavior is fully consistent with the phase drift observed in Fig~\ref{fig:b10_1dN}.}

However, this discrepancy should not be misinterpreted as a failure of the mean-field approximation. Instead, it reveals the extreme \textit{numerical stiffness} of the particle dynamics during synchronization. When the time step is refined significantly beyond the standard scaling, the particle system converges precisely to the PDE limit. As demonstrated in Table \ref{tab_time_b10_smalldt} and Figure \ref{fig:b10_16e3_1_small_dt}, with a fixed $N_p=16,000$ and an extremely fine time step $\delta t = 1/(1.6\times10^7)$, the particle solver reproduces the PDE blow-up times and mean periods with negligible error. {This improvement is clearly reflected in the quantitative data of Table~\ref{tab_time_b10_smalldt}. Compared with Table~\ref{tab_time_b10_dt_1dN}, both the event-wise timing error and the corresponding MSE are reduced by a substantial margin, while the average period of the particle system becomes nearly indistinguishable from that of the PDE solver. Therefore, Table~\ref{tab_time_b10_smalldt} confirms that the discrepancy seen under the standard scaling is primarily a temporal discretization artifact, rather than a breakdown of the mean-field description itself.} Similarly, Figure \ref{fig:b10_16e3_small_dt} confirms that the trajectories align perfectly when temporal resolution is sufficient.

This comparison highlights a critical computational bottleneck for particle methods: to correctly resolve the fast cascade during synchronization, the time step must be prohibitively small, independent of the population size. In contrast, the PDE solver, leveraging the time-dilation transformation, resolves these singular events efficiently without strictly constrained time steps {($\delta t_{\rm{pde}} = 10^{-4}$)}.

\paragraph{State-Dependent Finite Size Effects}
Next, we examine the regime of small population sizes to understand the nature of the finite size effect. Figure \ref{fig:b10_1000_100seeds} displays the ensemble-averaged trajectory over $100$ runs for $N_p = 1000$ with a sufficiently small time step. A key observation is that the deviation from the PDE solution is highly non-uniform in time.
\begin{itemize}
    \item \textbf{Inter-burst intervals:} Between synchronization events, where the network activity $N(t)$ is low, the particle simulation aligns closely with the PDE solver, suggesting a small finite size effect.
    \item \textbf{Synchronization events:} As the system approaches a blow-up and $N(t)$ diverges, the discrepancy widens significantly.
\end{itemize}
This dynamic behavior supports our conjecture that the finite size effect is state-dependent, scaling with the instantaneous network activity. The fluctuations are "amplified" during the synchronization, making the mean-field approximation harder to recover exactly during the burst without a massive increase in $N_p$.

\paragraph{Summary}
In summary, the periodic synchronization regime reveals two distinct advantages of the PDE solver. First, it eliminates the {stiffness penalty}, as it does not require the microscopic time steps needed by the particle solver to capture the cascade. Second, it eliminates the {amplified finite size noise} during synchronization, providing a clean description of the limit cycle that would otherwise require computationally infeasible particle numbers to resolve. 

\begin{figure}[H]
    \centering
    \includegraphics[width=0.8\linewidth]{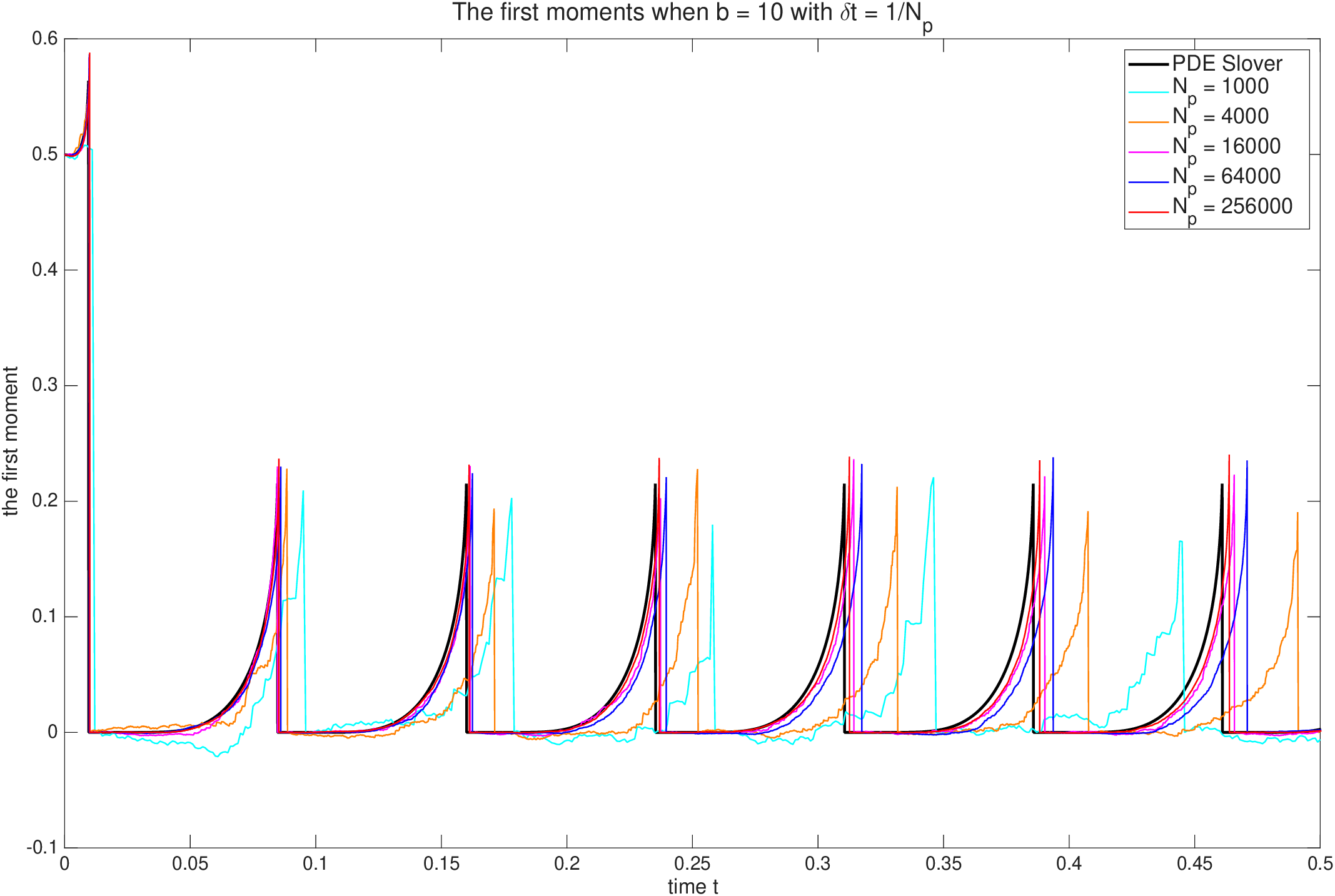}
    \caption{First-order moment comparison for $b = 10$ with standard scaling $\delta t = 1/N_p$. Note the phase drift accumulating over time even for large $N_p$, indicating insufficient temporal resolution in the particle solver.}
    \label{fig:b10_1dN}
\end{figure}

\begin{figure}[H]
    \centering
    \includegraphics[width=0.8\linewidth]{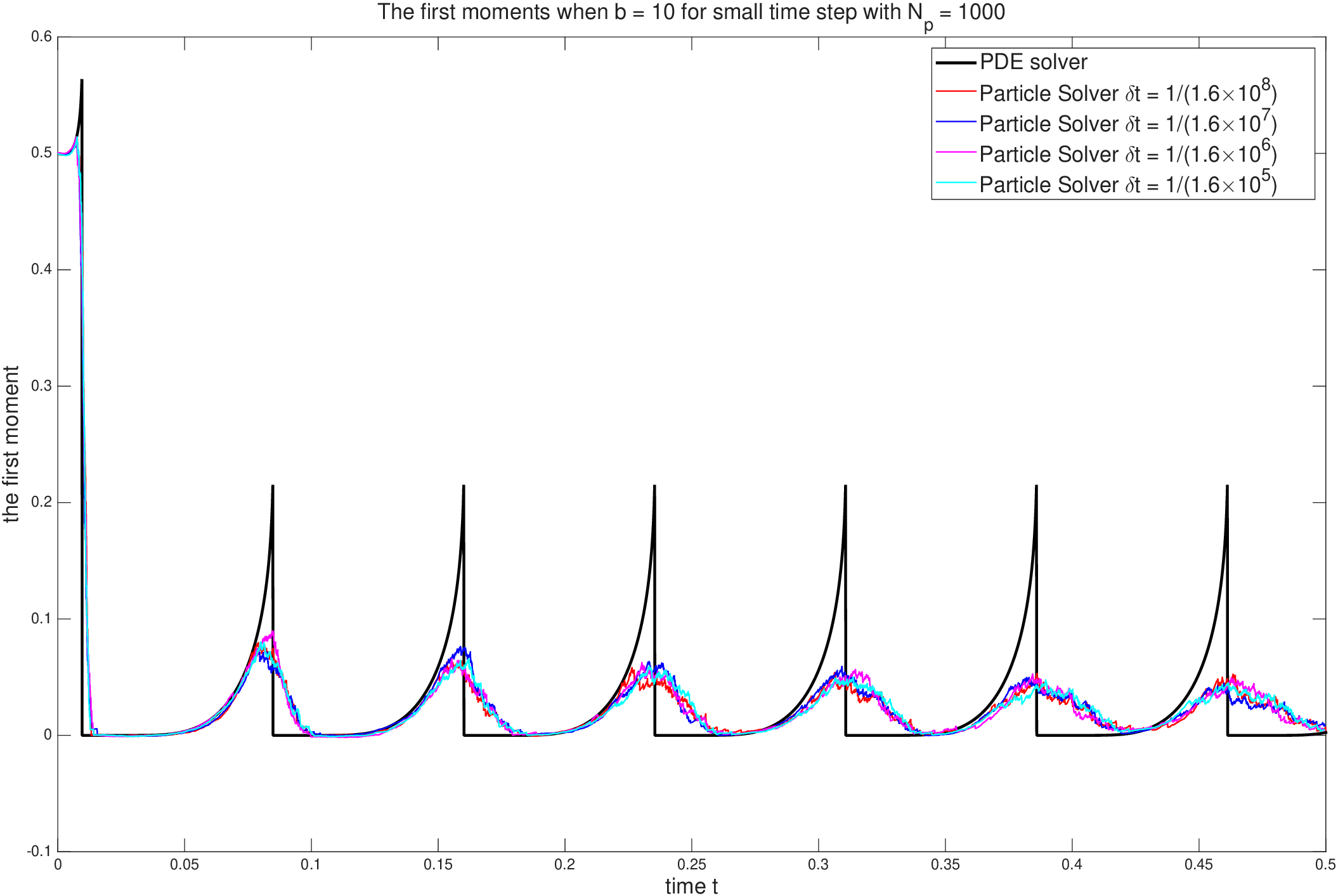}
    \caption{Ensemble-averaged first-order moment for $N_p = 1,000$ ($b=10$). The discrepancy (finite size effect) is minimal during the quiescent periods but is significantly amplified during the synchronization events.}
    \label{fig:b10_1000_100seeds}
\end{figure}

\begin{figure}[H]
    \centering
    \includegraphics[width=0.8\linewidth]{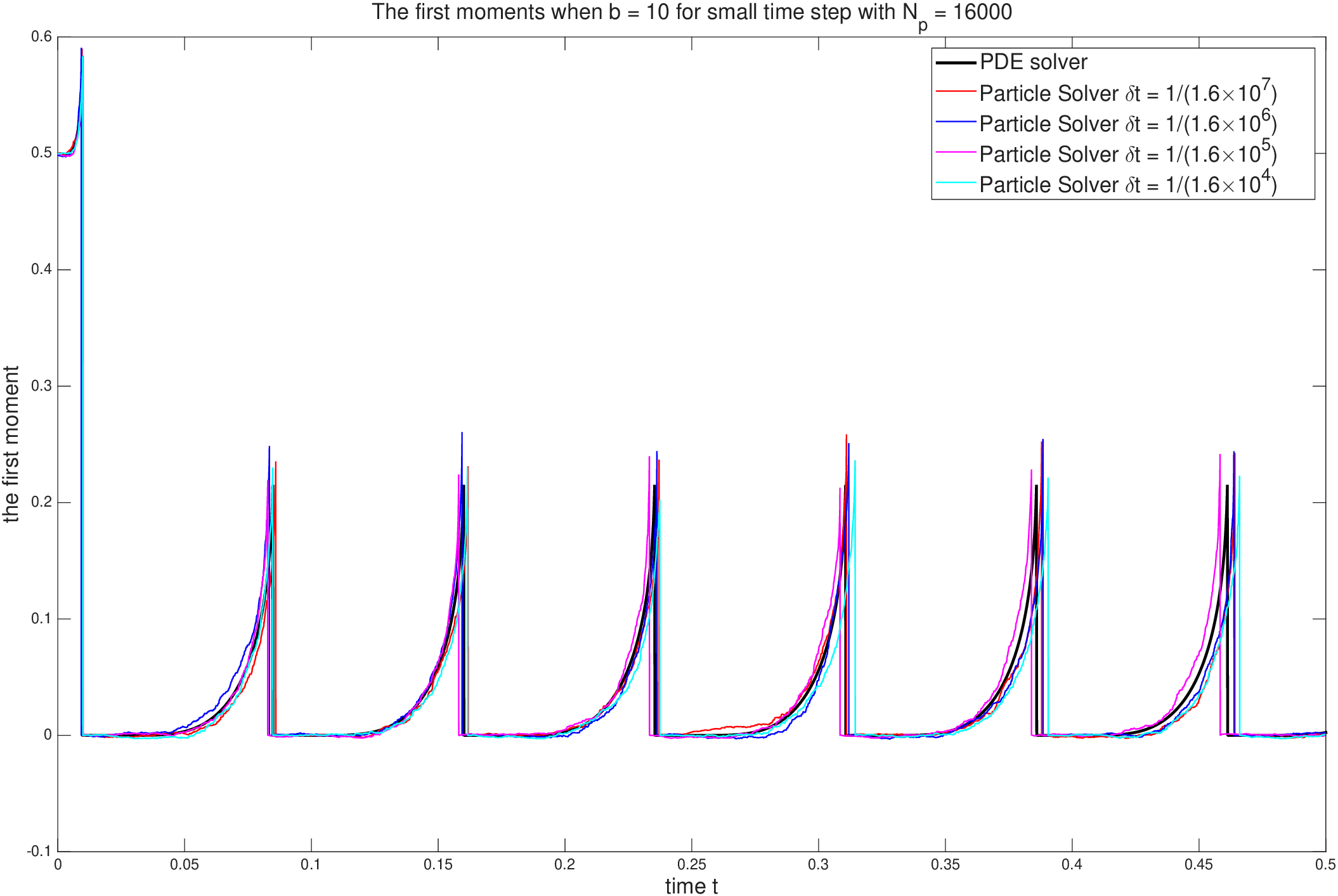}
    \caption{Comparison for $N_p = 16,000$ ($b=10$) with various small time steps. The convergence to the PDE solution improves as the time step decreases, confirming the stiffness of the system.}
    \label{fig:b10_16e3_small_dt}
\end{figure}

\begin{figure}[H]
    \centering
    \includegraphics[width=0.8\linewidth]{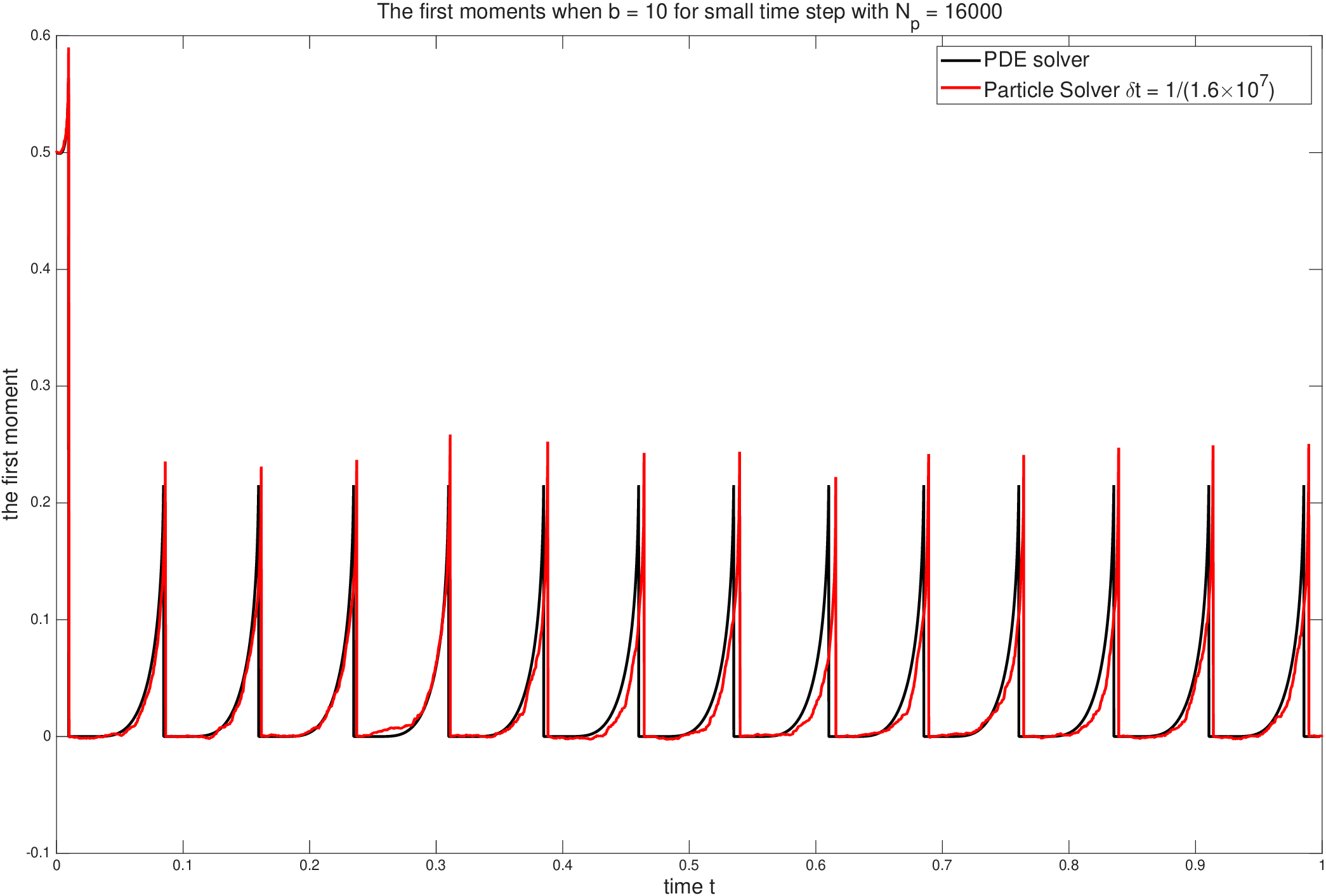}
    \caption{High-precision comparison for $N_p = 16,000$ ($b=10$) with $\delta t = 1/(1.6\times10^7)$. Under sufficiently fine temporal resolution, the particle system perfectly matches the PDE limit. For the particle solver, only the midpoints of the histogram bins are shown, with bin edges spaced at $0.05$.}
    \label{fig:b10_16e3_1_small_dt}
\end{figure}



\begin{table}[!htbp]
    \centering
    \caption{{Comparison of blow-up times ($b=10$) for fixed $N_p = 16,000$ with time steps $\delta t = 1/N_p$. The particle statistics are computed from $100$ independent runs. The average period of the PDE system is $7.52 \times 10^{-2}$, while the average period obtained from the particle system is $7.66 \times 10^{-2}$.}}
    \begin{tabular}{|c|c|c|c|c|c|}
    \hline
    {} &{5th time} &{6th time} & {7th time} & {8th time} & {9th time}\\
    \hline
    {PDE} & {$3.10\times10^{-1}$} & {$3.86\times10^{-1}$} & {$4.61 \times 10^{-1}$} & {$5.36 \times 10^{-1}$} & {$6.11 \times 10^{-1}$} \\ 
    \hline
    {PS mean} & {$3.17\times10^{-1}$} & {$3.94\times10^{-1}$} & {$4.70 \times 10^{-1}$} & {$5.47 \times 10^{-1}$} & {$6.23 \times 10^{-1}$}  \\
    \hline
    {PS variance} & {$8.27\times10^{-6}$} & {$9.71\times10^{-6}$} & {$1.25\times10^{-5}$} & {$1.52\times10^{-5}$} & {$1.73\times10^{-5}$}\\ 
    \hline
    {MSE} & {$5.07\times10^{-5}$} & {$7.26\times10^{-5}$} & {$9.75\times10^{-5}$} & {$1.34\times10^{-4}$} & {$1.63\times10^{-4}$}  \\ 
    \hline 
    \end{tabular}
    \label{tab_time_b10_dt_1dN}
\end{table}



\begin{table}[!htbp]
    \centering
    \caption{{Comparison of blow-up times ($b=10$) for fixed $N_p = 16,000$ with  time steps $\delta t = 1/(1.6\times10^7)$. The particle statistics are computed from $100$ independent runs. The average period of the PDE system is $7.52 \times 10^{-2}$, while the average period obtained from the particle system is $7.55 \times 10^{-2}$.}}
    \begin{tabular}{|c|c|c|c|c|c|}
    \hline
    {} &{5th time} &{6th time} & {7th time} & {8th time} & {9th time}\\
    \hline
    {PDE} & {$3.10\times10^{-1}$} & {$3.86\times10^{-1}$} & {$4.61 \times 10^{-1}$} & {$5.36 \times 10^{-1}$} & {$6.11 \times 10^{-1}$} \\ 
    \hline
    {PS mean} & {$3.12\times10^{-1}$} & {$3.87\times10^{-1}$} & {$4.63 \times 10^{-1}$} & {$5.38 \times 10^{-1}$} & {$6.14 \times 10^{-1}$}  \\
    \hline
    {PS variance} & {$9.55\times10^{-6}$} & {$1.02\times10^{-5}$} & {$1.26\times10^{-5}$} & {$1.26\times10^{-5}$} & {$1.58\times10^{-5}$}\\ 
    \hline
    {MSE} & {$1.08\times10^{-5}$} & {$1.25\times10^{-5}$} & {$1.58\times10^{-5}$} & {$1.73\times10^{-5}$} & {$2.23\times10^{-5}$}  \\ 
    \hline 
    \end{tabular}
    \label{tab_time_b10_smalldt}
\end{table}

\subsubsection{Periodic Synchronization with Moderate Coupling ($b=2$)}\label{subsecb2}

We next examine the case of moderate coupling ($b = 2$). In this regime, the system exhibits periodic synchronization similar to the $b=10$ case, but with significantly wider inter-burst intervals (mean period $\approx 0.21$). While the dynamics appear milder, our analysis reveals that the computational challenges for the microscopic particle simulation---specifically the numerical stiffness and the finite size effect---persist and remain non-negligible.

\paragraph{Phase Drift and Numerical Stiffness}
We first evaluate the performance of the particle solver using the standard time step scaling $\delta t = 1/N_p$. As shown in Figure \ref{fig:b2_1dN} and Table \ref{tab_time_b2_dt_1dN}, despite the lower frequency of blow-ups, the particle simulation suffers from a cumulative phase drift. Even with a large population of $N_p = 256,000$, the predicted blow-up times begin to deviate from the mean-field limit after the second cycle.

This discrepancy highlights the inherent stiffness of the synchronization mechanism. The ``avalanche" of spikes occurs on a timescale much faster than the inter-burst relaxation. Consequently, the standard step $\delta t = 1/N_p$ is insufficient to accurately resolve the precise timing of the cascade, leading to a delay that accumulates over each period. {This trend is also quantified in Table~\ref{tab_time_b2_dt_1dN}. Even though the moderate-coupling regime has a much longer inter-burst interval than the case $b=10$, the particle solver with the standard scaling still exhibits a systematic delay in the predicted blow-up times. The discrepancy again accumulates from one event to the next, showing that the stiffness associated with the cascade mechanism remains present even when the bursts are less frequent.}

To confirm that this is a numerical artifact rather than a model discrepancy, we refine the time step while holding the population size fixed at $N_p = 16,000$. As presented in Table \ref{tab_time_b2_smalldt}, reducing the time step to $\delta t = 1/(1.6\times10^7)$ significantly reduces the discrepancy in blow-up times 
and aligns the mean period closely with the PDE prediction. {A direct comparison between Table~\ref{tab_time_b2_dt_1dN} and Table~\ref{tab_time_b2_smalldt} shows that this phase drift is greatly suppressed once the time step is sufficiently refined. In particular, the event-wise blow-up times become much closer to the PDE values, and the remaining discrepancy stays at a much smaller level over multiple periods. This again demonstrates that, even in the moderate-coupling regime, resolving the fast cascade dynamics requires a time step far smaller than the conventional scaling $\delta t = 1/N_p$.} Figure \ref{fig:b2_16e3_small_dt} further confirms that with sufficiently fine temporal resolution, the particle trajectory recovers the phase coherence of the mean-field limit.

\paragraph{Amplification of Finite Size Effects}
The state-dependence of the finite size effect is clearly observed in the small population regime. Figure \ref{fig:b2_1000_100seeds} displays the ensemble-averaged dynamics {over $100$ runs} for $N_p = 1000$. Similarly to the strong coupling case, the particle model agrees well with the PDE solution during the quiescent phases but exhibits significant deviations near the synchronization events. This confirms that the finite size noise is not additive but is dynamically amplified by the network activity, necessitating much larger population sizes to control the error during blow-ups compared to the steady phases.

\paragraph{Conclusion}
The results for $b=2$ reinforce the conclusion that the computational advantage of the PDE solver is universal across periodic regimes. Even when the synchronization is less frequent, the particle method requires disproportionately expensive time steps ($\sim 10^{-7}$) to prevent phase drift. The PDE solver, by contrast, naturally handles these multiscale dynamics through the time-dilation transformation, providing an efficient and accurate description of the limit cycle.


\begin{figure}[H]
    \centering
    \includegraphics[width=0.8\linewidth]{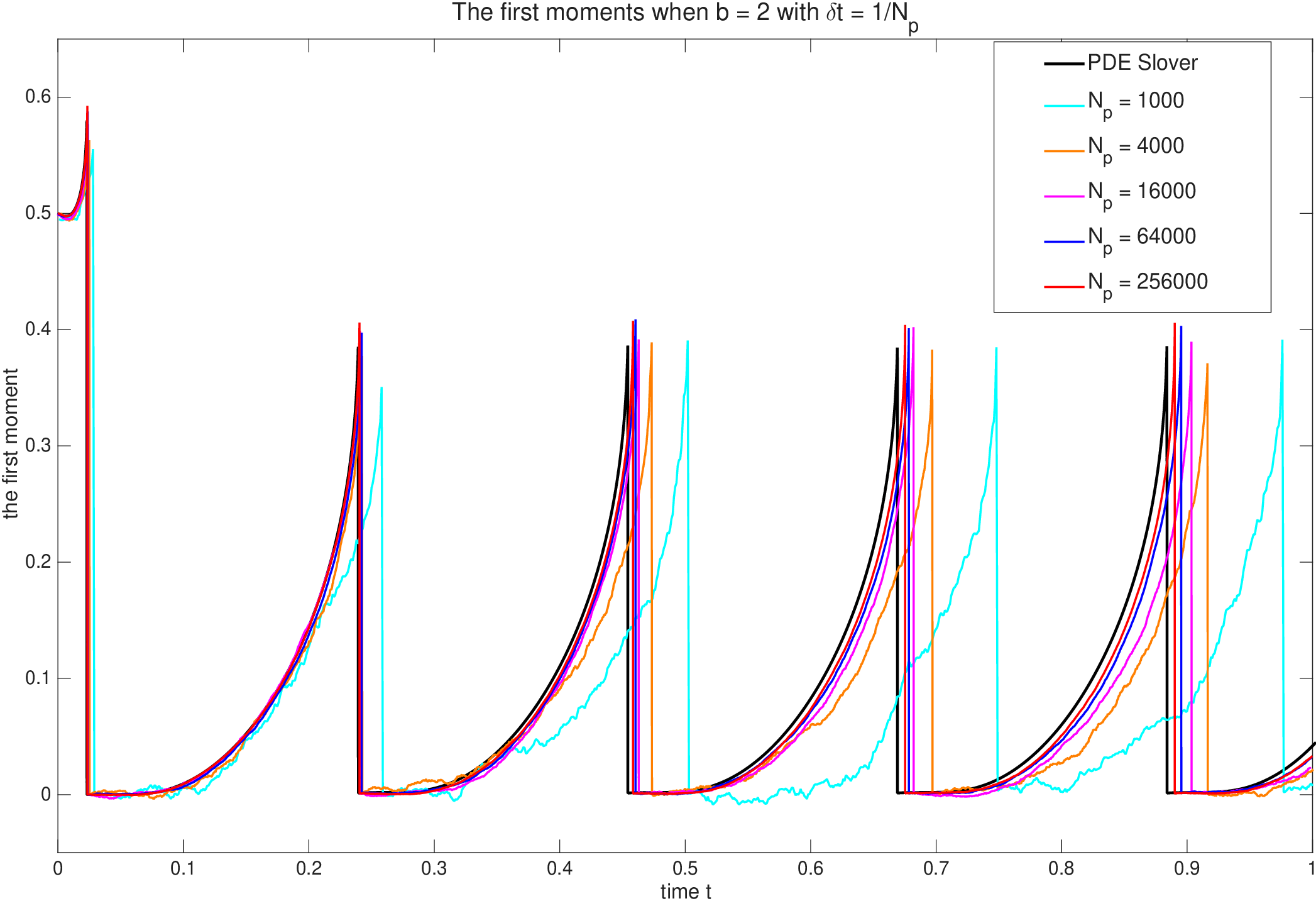}
    \caption{Comparison of first-order moment trajectories for $b = 2$ using standard scaling $\delta t = 1/N_p$. Note that phase drift occurs even for large $N_p$, indicating that the standard time step is insufficient to resolve the stiffness of the blow-up events.}
    \label{fig:b2_1dN}
\end{figure}

\begin{figure}[H]
    \centering
    \includegraphics[width=0.8\linewidth]{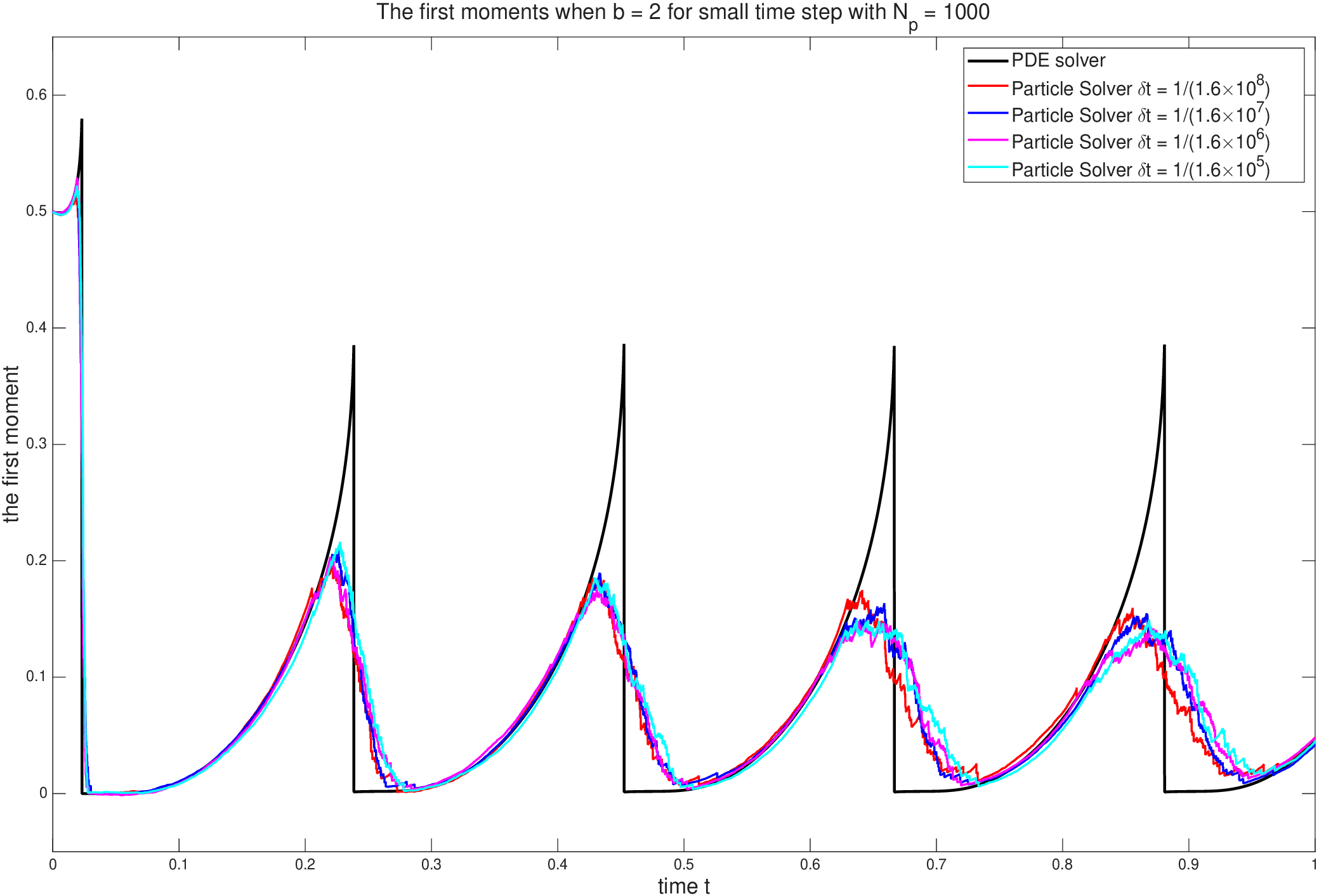}
    \caption{Ensemble-averaged first-order moment for $N_p = 1000$ ($b=2$). The finite size fluctuations are state-dependent, remaining small during the relaxation phase but expanding significantly during synchronization.}
    \label{fig:b2_1000_100seeds}
\end{figure}

\begin{figure}[H]
     \centering
    \includegraphics[width=0.8\linewidth]{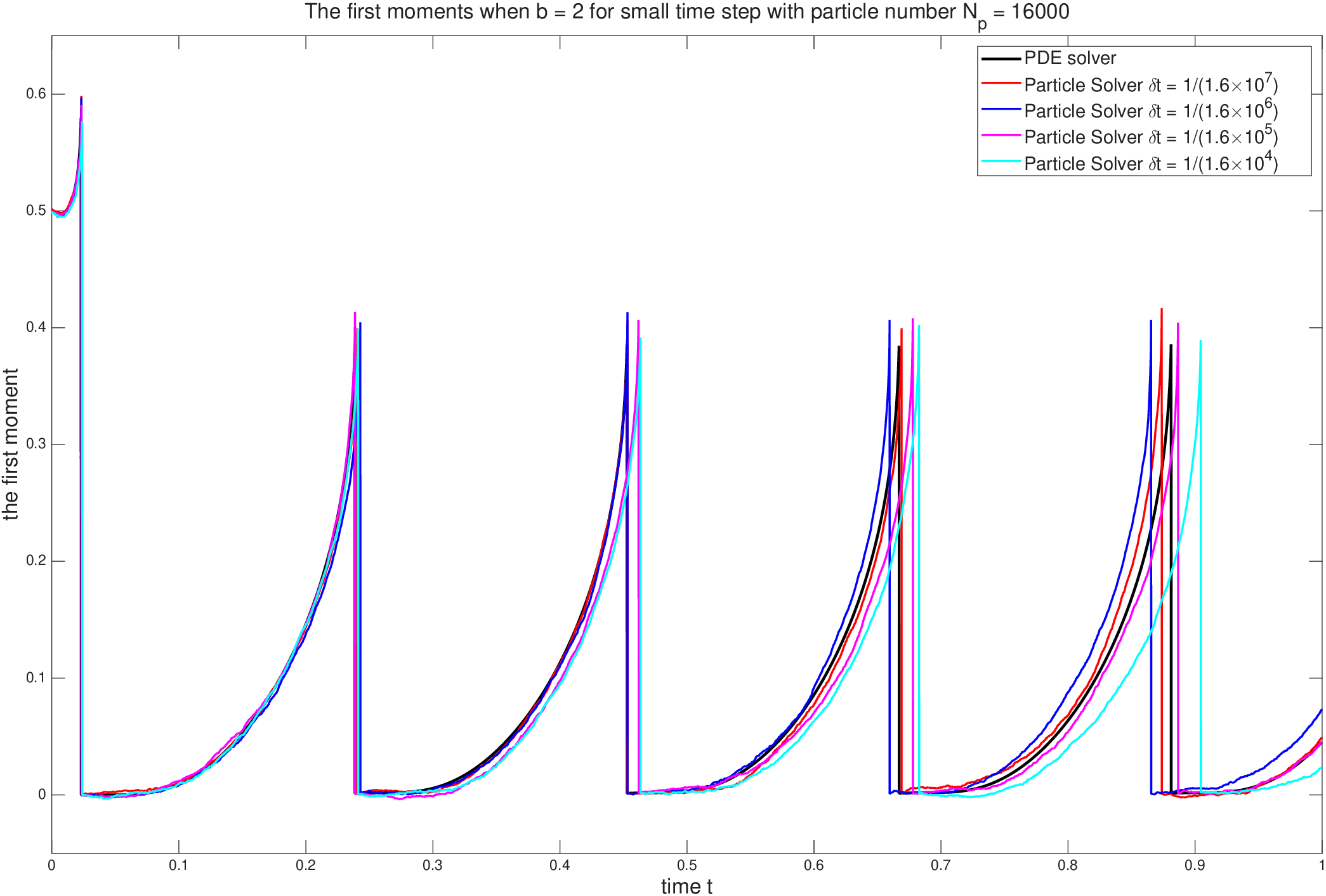}
     \caption{Comparison for $N_p = 16,000$ ($b = 2$) with extremely small time steps. With sufficient temporal resolution to overcome stiffness, the particle model faithfully tracks the mean-field PDE solution.}
     \label{fig:b2_16e3_small_dt}
\end{figure}



\begin{table}[!htbp]
    \centering
    \caption{{Comparison of blow-up times ($b=2$) for fixed $N_p = 16,000$ with time steps $\delta t = 1/N_p$. The particle statistics are computed from $100$ independent runs. The average period of the PDE system is $2.14 \times 10^{-1}$, while the average period obtained from the particle system is $2.20 \times 10^{-1}$.}}
    \begin{tabular}{|c|c|c|c|c|c|}
    \hline
    {} &{5th time} &{6th time} & {7th time} & {8th time} & {9th time}\\
    \hline
    {PDE}  & {$8.81\times10^{-1}$} & {$1.10$} & {$1.31$} & {$1.52$} & {$1.74$} \\ 
    \hline
    {PS mean} & {$9.04\times10^{-1}$} & {$1.12$} & {$1.34$} & {$1.56$} & {$1.78$}  \\
    \hline
    {PS variance} & {$5.78\times10^{-5}$} & {$6.84\times10^{-5}$} & {$8.27\times10^{-5}$} & {$9.82\times10^{-5}$} & {$1.22\times10^{-4}$}\\ 
    \hline
    {MSE} & {$5.57\times10^{-4}$} & {$8.28\times10^{-4}$} & {$1.14\times10^{-3}$} & {$1.56\times10^{-3}$} & {$2.02\times10^{-3}$}  \\ 
    \hline 
    \end{tabular}
    \label{tab_time_b2_dt_1dN}
\end{table}
\begin{table}[!htbp]
    \centering
    \caption{{Comparison of blow-up times ($b=2$) for fixed $N_p = 16,000$ with time steps $\delta t = 1/(1.6\times10^7)$. The particle statistics are computed from $100$ independent runs. The average period of the PDE system is $2.14 \times 10^{-1}$, while the average period obtained from the particle system is $2.15 \times 10^{-1}$.}}
    \begin{tabular}{|c|c|c|c|c|c|}
    \hline
    {}  &{5th time} & {6th time} & {7th time} & {8th time} &{9th time}\\
    \hline
    {PDE}  & {$8.81\times10^{-1}$} & {$1.10$} & {$1.31$} & {$1.52$} & {$1.74$}\\ 
    \hline
    {PS mean}  & {$8.86\times10^{-1}$} & {$1.11$} & {$1.32$} & {$1.53$} & {$1.75$} \\
    \hline
    {PS variance}  & {$6.86\times10^{-5}$} & {$6.80\times10^{-5}$} & {$8.10\times10^{-5}$} & {$1.15\times10^{-4}$} & {$1.57\times10^{-4}$}\\ 
    \hline
    {MSE}  & {$8.77\times10^{-5}$} & {$9.21\times10^{-5}$} & {$1.16\times10^{-4}$} & {$1.67\times10^{-4}$} & {$2.06\times10^{-4}$}  \\ 
    \hline 
    \end{tabular}
    \label{tab_time_b2_smalldt}
\end{table}

\subsubsection{Critical Regime: Single Blow-up ($b = 1$)}

Finally, we investigate the critical regime ($b = 1$), where the system, under the chosen initial conditions, sits precisely on the threshold of a single synchronization event. This regime imposes the most stringent test on the solvers, as the dynamics are highly sensitive to both numerical precision and finite-size fluctuations. The PDE solver predicts a sharp single blow-up at $t \approx 0.07$.

\paragraph{Failure of Standard Particle Simulations}
We first attempt to reproduce this event using the particle solver with the standard scaling $\delta t = 1/N_p$. As shown in Figure \ref{fig:b1_dt_1dN} and Table \ref{tab_b1_dt_1dN}, the particle simulation fails to undergo synchronization entirely, even when the population size is increased to $N_p = 256,000$. The computed first-order moment remains smooth, completely missing the sharp transition predicted by the mean-field limit. This suggests that near the critical threshold, standard Monte Carlo methods are incapable of capturing the rare collective event due to insufficient temporal resolution or noise interference.

\paragraph{The Finite Size Barrier}
To determine if this failure is physical (due to system size) or numerical (due to time step), we first examine a small system ($N_p = 1000$) with extremely fine time steps. Figure \ref{fig:b1_1000_100seeds} displays the ensemble-averaged trajectory over 100 runs. Even with $\delta t$ as small as $10^{-8}$, the small system does not synchronize. This indicates a dominant {finite size effect}: for small populations, the intrinsic stochastic fluctuations prevent the system from crossing the critical synchronization threshold, resulting in qualitatively different physics compared to the mean-field limit.

\paragraph{Extreme Stiffness and Computational Cost}
We then test a larger system ($N_p = 16,000$) where the mean-field approximation should ideally hold. Figure \ref{fig:b1_16000_small_dt} reveals the extreme numerical stiffness of the problem. With time steps of $10^{-5}$ or even $10^{-7}$, the particle solver still misses the blow-up. It is only when the time step is refined to an extreme {$\delta t = 1/(1.6\times10^8)$}  that the particle system finally reproduces the single blow-up event.

\paragraph{Conclusion}
This case explicitly demonstrates the superiority of the PDE solver in critical regimes. While the particle method requires a ``perfect storm" of large $N_p$ and infinitesimal $\delta t$ (incurring a massive computational cost) to trigger the correct dynamics, the PDE solver captures the critical blow-up naturally and efficiently.


\begin{figure}[H]
    \centering
    \includegraphics[width=0.8\linewidth]{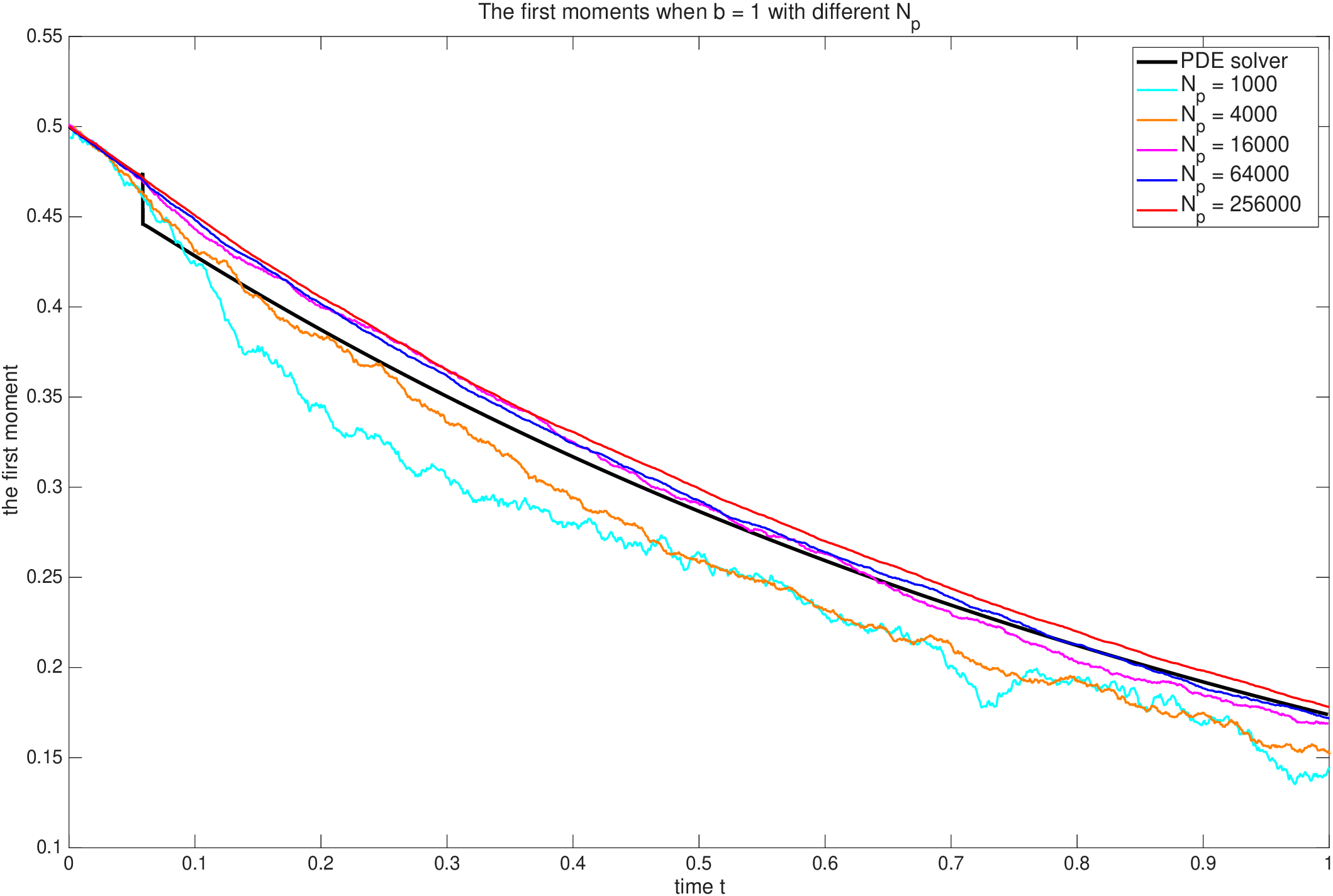}
    \caption{Comparison for the critical case $b = 1$ using standard scaling $\delta t = 1/N_p$. The particle solver (colored lines) completely misses the single blow-up event predicted by the PDE solver (black dashed line), even for $N_p = 256,000$.}
    \label{fig:b1_dt_1dN}
\end{figure}

\begin{table}[htbp]
    \centering
    \caption{$L_{\rm{error}}$ for $b = 1$ with varying $N_p$ and $\delta t = 1/N_p$. The particle statistics are computed from $100$ independent runs. The large errors reflect the qualitative failure of the particle solver to capture the synchronization event.}
    \begin{tabular}{|c|c|c|c|c|c|}
    \hline
    {$N_p$} & {$1000$} & {$4000$} & {$16,000$} & {$64,000$} & {$256,000$} \\ 
    \hline
    {$L_{\rm{error}}$ mean} & {$4.13\times10^{-2}$} & {$1.86\times10^{-2}$} & {$7.59\times10^{-3}$} & {$6.70\times10^{-3}$} & {$1.05\times10^{-2}$} \\
    \hline
    {$L_{\rm{error}}$ var} & {$3.67\times10^{-4}$} & {$5.66\times10^{-5}$} & {$7.37\times10^{-6}$} & {$2.03\times10^{-6}$} & {$1.22\times10^{-6}$} \\
    \hline
    \end{tabular}
    \label{tab_b1_dt_1dN}
\end{table}

\begin{figure}[H]
    \centering
    \includegraphics[width=0.8\linewidth]{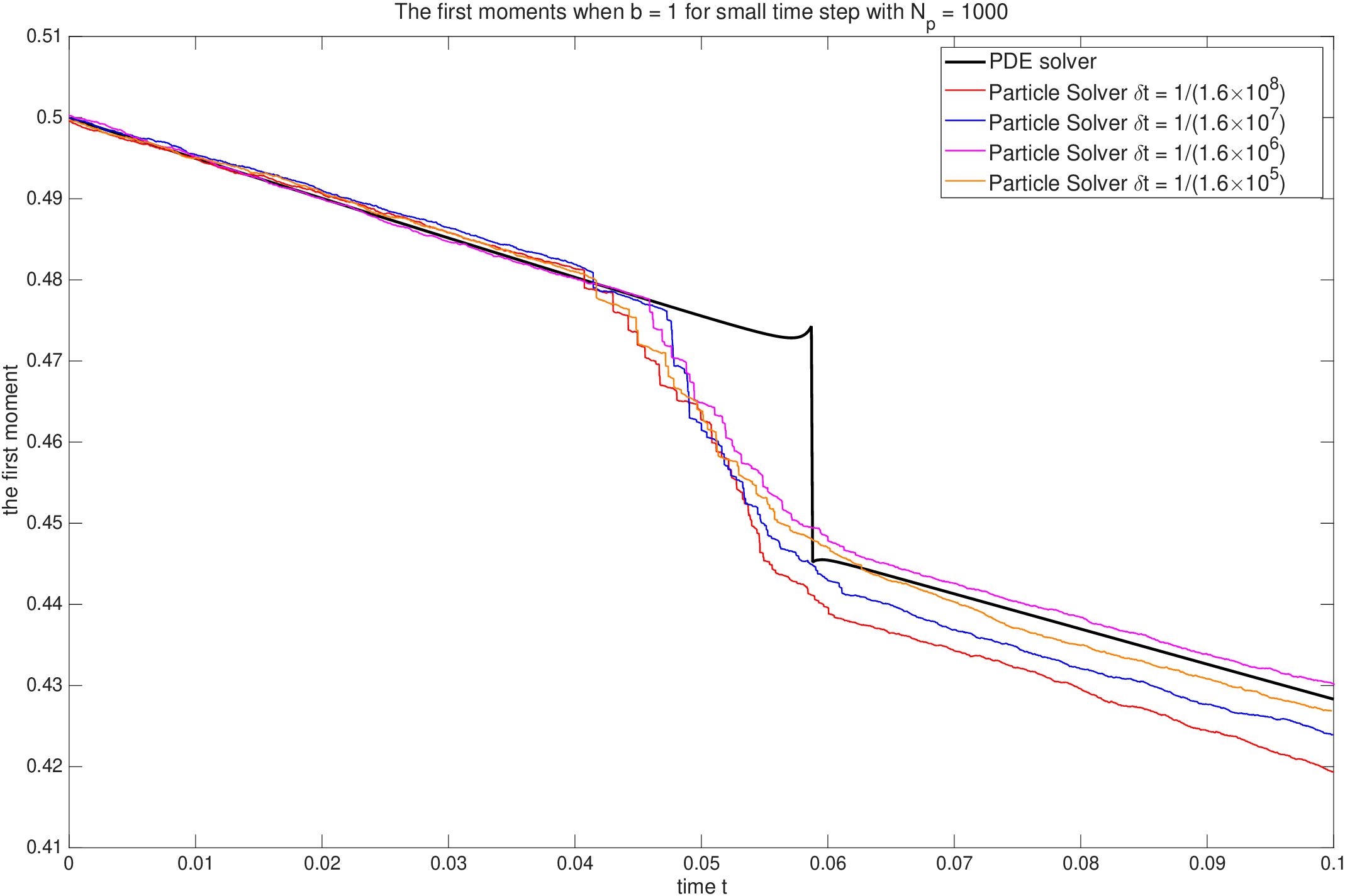}
    \caption{Ensemble-averaged first-order moment for $N_p = 1000$ ($b = 1$) with extremely small time steps. The system fails to synchronize, indicating a physical barrier imposed by the finite size effect.}
    \label{fig:b1_1000_100seeds}
\end{figure}

\begin{figure}[H]
    \centering
    \includegraphics[width=0.8\linewidth]{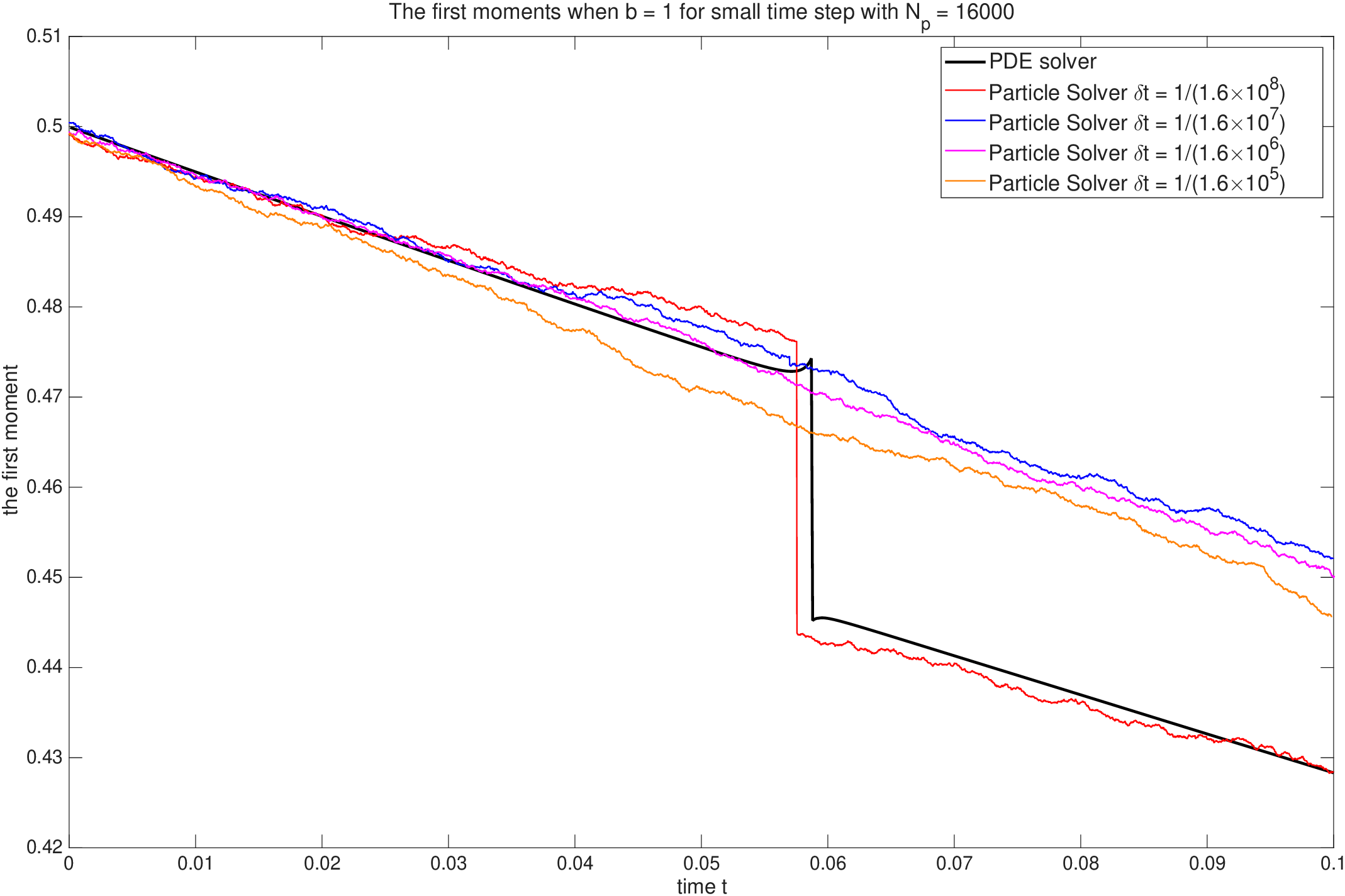}
    \caption{Comparison for $N_p = 16,000$ ($b = 1$) with decreasing time steps. The particle system captures the blow-up only at the extreme time step of $\delta t = 1/(1.6\times10^8)$, highlighting the severe numerical stiffness.}
    \label{fig:b1_16000_small_dt}
\end{figure}

\subsection{Discussion: Computational Advantages and Finite-Size Limitations}

The numerical experiments presented in this section serve a dual purpose: validating the proposed multiscale PDE algorithm and exploring the physical relationship between the microscopic particle model and its macroscopic mean-field limit. Here, we summarize the capabilities of our solver and delineate the regimes where the mean-field approximation is most effective versus where finite-size effects dominate.

\paragraph{Capabilities of the Time-Dilation Solver}
The primary achievement of this work is the robust resolution of finite-time blow-ups. Our results confirm that the time-dilation framework successfully transforms the singular firing rate $N(t)$ into a regularized auxiliary variable $M(\tau)$. This allows the solver to:
\begin{enumerate}
    \item \textbf{Seamlessly Traverse Singularities:} The solver automatically detects the onset of synchronization and switches timescales, calculating the blow-up size $\Delta m$ and the post-blowup distribution without manual intervention or heuristic regularization.
    \item \textbf{Capture Complex Dynamics:} From steady states ($b=1/2$) to high-frequency periodic oscillations ($b=10$), the solver maintains stability and mass conservation, reproducing the full spectrum of theoretical solutions.
\end{enumerate}

\paragraph{Computational Efficiency and Stiffness}
When regarding the PDE solver as a proxy for large-scale neuronal networks ($N_p \gg 1$), our comparisons highlight a decisive computational advantage. The microscopic dynamics during a synchronization event are extremely \textit{stiff}, requiring particle simulations to employ infinitesimal time steps (down to $\delta t \sim 10^{-8}$) to resolve the fast cascade. 
The PDE solver eliminates this stiffness through the change of variables $d\tau = N(t)dt$. Consequently, it can produce noise-free, highly accurate descriptions of the collective dynamics using standard grid sizes. This makes the PDE approach orders of magnitude more efficient for studying the phase diagrams and synchronization properties of large networks.

\paragraph{Finite-Size Effects and Physical Discrepancies}
However, it is crucial to acknowledge that the mean-field PDE describes the limit $N_p \to \infty$. Our experiments with the particle system reveal significant \textit{finite-size effects} that distinguish the microscopic reality from the macroscopic limit:
\begin{itemize}
    \item \textbf{State-Dependent Fluctuations:} The deviation between the particle system and the PDE is not uniform. Fluctuations are minimal during quiescent periods but are dynamically amplified during synchronization events.
    \item \textbf{Critical Barriers:} In critical regimes (e.g., $b=1$), the finite system size imposes a physical barrier. While the mean-field limit predicts a deterministic blow-up, small particle populations ($N_p=1000$) may fail to synchronize due to intrinsic stochasticity. In such cases, the PDE solver accurately describes the \textit{infinite-population potential} for synchronization, but the realization of this potential in a finite network depends on the specific population size. We note that the slow convergence of the mean-field limit when $b=1$ is also reported in \cite{sadun2022global} for a related model, under the term ``critical slowing down''.
\end{itemize}

In conclusion, the proposed PDE solver is a powerful tool for investigating the ideal collective behavior of integrate-and-fire networks, offering a precise baseline against which finite-size effects and stochastic phenomena can be measured.



\section{Conclusion}

In this paper, we have presented a robust multiscale numerical framework for simulating multiple firing events in the mean-field Integrate-and-Fire neuronal network model. By introducing a {time-dilated coordinate} proportional to the global firing activity, we successfully desingularized the finite-time blow-up inherent to the Fokker-Planck equation with strong excitatory feedback. This transformation allows us to resolve the instantaneous synchronization event as a continuous mesoscopic process, overcoming the severe time-step stiffness that plagues traditional particle-based simulations.

A key feature of our approach is the {hybrid numerical strategy}, which combines a standard finite-volume scheme for smooth evolution with a specialized treatment for the blow-up regime. The introduction of a mesh-independent flux criterion ensures that the solver switches timescales only when physically necessary, while the semi-analytical ``moving Gaussian'' method accurately captures the post-blowup distribution without numerical diffusion. Crucially, we demonstrated that this numerical regularization is physically consistent with the microscopic {avalanche mechanism}, providing a macroscopic description that respects the system's underlying fragility and conservation laws.

Several promising directions for future research emerge from this work. First, extending this framework to {Excitatory-Inhibitory (E-I) networks} would allow for the exploration of more complex rhythms, such as the interplay between synchronization and inhibition-induced oscillations. Second, incorporating {synaptic delays} into the mean-field model presents an interesting challenge, as delays introduce history dependence that may interact non-trivially with the synchronous dynamics. Finally, from a theoretical perspective, establishing a rigorous error analysis and convergence rate for the hybrid time-dilation scheme remains an open mathematical problem that warrants further investigation. {Note that even at the fully continuous level, the global existence of the generalized solution here is open, as previous results do not directly apply due to the infinitesimal refractory period.}

\section*{Acknowledgments}

ZZ is supported by the National Key R\&D Program of China, Project Number 2021YFA1001200, Zhejiang Provincial Natural Science Foundation of China, Project Number QKWL25A0501, and Fundamental and Interdisciplinary Disciplines Breakthrough Plan of the Ministry of Education of China, Project Number JYB2025XDXM502. LT is supported in part by the Brain Science and Brain-like Intelligence Technology Innovation STI2030 -- National Science and Technology Major Project No. 2022ZD0204600. XD and ZZ thank Jian-Guo Liu and Jos\'{e} A. Carrillo for helpful discussions. XD thanks Yangfan Luo for perspectives on convection-diffusion problems.

\bibliography{mfe.bib}
\bibliographystyle{plain}

\end{document}